\documentclass[11pt]{article}
\usepackage{amsmath}
\usepackage{amssymb}
\usepackage{latexsym}
\usepackage{amsthm}
\usepackage{eucal}
\usepackage{mathrsfs}
\usepackage{wasysym}
\usepackage{fancyhdr}
\usepackage{amsfonts}
\usepackage{amssymb}
\usepackage{epsfig}
\usepackage{epstopdf}
\usepackage[normalem]{ulem}

\usepackage{tkz-graph}
\usetikzlibrary{arrows}
\usepackage{tkz-euclide}
\usetkzobj{all}

\newtheorem{thm}{Theorem}[section]
\newtheorem{cor}[thm]{Corollary}
\newtheorem{lemma}[thm]{Lemma}

\begin{document}

\title{Cycle Traversability  for Claw-free Graphs and Polyhedral Maps}

\author{
Ervin Gy\H{o}ri\thanks{Alfr\'ed R\'enyi Institute of Mathematics, Hungarian Academy of Sciences, H-1053 Budapest, Re\'altanoda u. 13-15, Hungary. Email: {\tt gyori@renyi.hu}. Partially supported by the NKFIH Grant K116769.}, \, 
Michael D. Plummer\thanks{Department of Mathematics, Vanderbilt University, Nashville, TN 37240, USA. Email: {\tt michael.d.plummer@vanderbilt.edu}.}, \,
Dong Ye\thanks{ Department of Mathematical Sciences, Middle Tennessee State University, Murfreesboro, TN 37132, USA. Email: {\tt dong.ye@mtsu.edu}. Partially supported by a grant from the Simons Foundation (No. 359516).} \, and  
Xiaoya Zha\thanks{Department of Mathematical Sciences,
Middle Tennessee State University, Murfreesboro, TN 37132, USA. Email: {\tt xiaoya.zha@mtsu.edu}.}
}

\date{July 3 2018}

\maketitle

\begin{abstract}
Let $G$ be a graph, and $v\in V(G)$ and $S\subseteq V(G)\backslash v$ of size at least $k$. 
An important result on graph connectivity due to Perfect states that, if $v$ and $S$ are $k$-linked, then a $(k-1)$-link between a vertex $v$ and $S$  can be extended to a $k$-link between $v$ and $S$ such that the endvertices of the $(k-1)$-link are also the endvertices of the $k$-link. We begin by proving a generalization of Perfect's result by showing that, if two disjoint sets $S_1$ and $S_2$ are $k$-linked, then a $t$-link ($t< k$) between two disjoint sets $S_1$ and $S_2$ can be extended to a $k$-link between $S_1$ and $S_2$ such that the endvertices of the $t$-link are preserved in the $k$-link.

Next, we are able to use these results to show that a 3-connected claw-free graph  always has a cycle passing through any given five vertices but avoiding any other one specified vertex. 
We also show that this result is sharp by exhibiting an infinite family of 3-connected claw-free graphs
in which there is no cycle containing a certain set of six vertices but avoiding a seventh specified vertex. 
A direct corollary of our main result shows that, a 3-connected claw-free graph has a topological wheel minor $W_k$ with $k\le 5$ if and only if it has a vertex of degree at least $k$.  

Finally, we also show that a graph polyhedrally embedded in a surface always has a cycle passing through any given three vertices but avoiding any other specified vertex. The result is best possible in the sense that the polyhedral embedding assumption is necessary, and there are infinitely many graphs polyhedrally embedded in any surface having no cycle containing a certain set of four vertices but avoiding a fifth specified vertex.   

\medskip


\end{abstract}

\section{Introduction}

A graph $G$ is Hamiltonian if $G$ has a cycle containing all the vertices of $G$.
Hamiltonicity of graphs is one of the major topics in graph theory.
High connectivity does not guarantee the existence of a Hamilton cycle in a graph,
but a highly connected graph does contain a long cycle. 
For example, given any $k$ vertices of a $k$-connected graph $G$, there is a cycle containing all $k$ of them.
(cf. \cite{D}).
Bondy and Lov\'asz \cite{BL} proved an even stronger result which says that for any given vertex set $S$ of size $k-1$ in a $k$-connected graph $G$, the cycles containing $S$ generate the cycle space of $G$. 
Besides the Hamiltonicity problem, Chv\'atal \cite{C} also considered the cyclability question for graphs, i.e., for a given set of vertices,
does there exist a cycle through these vertices.
Cyclability versus connectivity in graphs has been studied by a number of  authors (cf. \cite{K2,HM,FGLS}).


If one adds additional properties to the connectivity assumption, it is sometimes possible to guarantee higher cyclability. An example is the following nine-point theorem. 

\begin{thm}[Holton, Mckay, Plummer \& Thomassen, \cite{HMPT}]
Let $G$ be a 3-connected cubic graph. Then any set of nine vertices lies on a cycle. 
\end{thm}

The cyclability problem has also been studied for 3-connected graphs in the presence of additional properties.  Ellingham et al. \cite{EHL} showed that a 3-connected cubic graph has a cycle which passes through any ten given vertices if and only if the graph is not contractible to the Petersen graph in such a way that the ten vertices each map to a distinct vertex of the Petersen graph.

A great deal of attention had been paid to cycles through specified edges as well.  Lov\'asz \cite{L}, and independently, Woodall \cite{W} conjectured that every $k$-connected graph has a cycle through any $k$ given independent edges unless these edges form an odd edge-cut. H\"aggkvist and Thomassen \cite{HaT} proved a weak version of the Lov\'asz-Woodall conjecture -- that every $k$-connected graph has a cycle through any $k-1$ given independent edges, which was conjectured by Woodall \cite{W}. A complete proof of the Lov\'asz-Woodall conjecture was announced by Kawarabayashi \cite{K1}, but a complete proof has yet to appear.  Holton and Thomassen \cite{HoT} conjectured that any cyclically $(k+1)$-connected cubic graph has a cycle through any given $k$ independent edges, and this still remains open.  

The existence problem of a cycle through certain given edges in graphs is equivalent to the existence problem of a  cycle through corresponding vertices in their line graphs with certain forbidden pairs of edges incident with these vertices.  A graph is {\em claw-free} if it contains no induced subgraph isomorphic to $K_{1,3}$. 
Note that line graphs form a subfamily of claw-free graphs.
Hence, it is 
interesting to study cyclability of claw-free graphs. An analogue of the above nine-point theorem has been obtained for claw-free graphs by the first two authors of this paper \cite{GP}.

\begin{thm}[\cite{GP}]\label{thm:9-vtx-clawfree}
Let $G$ be a 3-connected claw-free graph and $S$ be a set with $k\le 9$ vertices.
Then $G$ has a cycle containing all the vertices of $S$.\par
\end{thm}

Recently, Chen \cite{ZC} show that a 3-connected claw-free graph $G$ has a cycle through any given twelve vertices if the underlying graph of its closure (a line graph) cannot be contracted to the Petersen graph in certain ways. 
In \cite{EHL}, Ellingham et. al. proved there is a cycle through any five vertices in a 3-connected cubic graph which avoids any specified edge.

In this paper, we consider the cyclability problem for graphs when certain sets of vertices are to be included and other sets are to be avoided. Let $G$ be any connected graph containing a cycle and $m$ and $n$, two non-negative integers with $1\le m+n\le |V(G)|$.
Then graph $G$ is said to satisfy {\em property $C(m,n)$} (or simply, $G$ is $C(m,n)$), if for  any two disjoint sets $S_1$ and $S_2$ contained in $V(G)$
with $|S_1|=m$ and $|S_2|=n$,
there is a cycle $C$ in $G$ such that $S_1\subseteq V(C)$, but $S_2\cap V(C)=\emptyset$.
When $n=0$, the maximum value of $m$ such that there is a cycle through every set $S\subseteq V(G)$ with $|S|\le m$
is known as the {\it cyclability} of $G$. Of course the case when $m=|V(G)|$ is just the well-studied Hamilton cycle problem. 
It has been shown in \cite{H, HP} that a graph $G$ is $k$-connected if and only if $G$ is $C(k-l, l)$ for all integer $0\le l\le k-2$, and in \cite{H,MW} that a graph  $G$ is $k$-connected if and only if $G$ is $C(2, k-2)$. So, in particular, a graph $G$ is 3-connected if and only if $G$ is $C(2,1)$. 

Cyclability problems and their near-relatives, have been widely studied for many graph classes.
We make no attempt at a comprehensive listing of results in this area here, but instead refer the reader to the recent survey of
Gould [G]. For the class of claw-free graphs there are many published results as well.
We refer the interested reader to \cite{FFR} which surveys results for claw-free graphs. Chudnovskey and Seymour recently 
published a deep analysis of claw-free graphs in a series of eight papers. (Cf. \cite{CS1} for more information.)   








Note that, a $k$-connected graph may not have a cycle through any given $k$ vertices which avoids a specified vertex. For example, the complete bipartite graph $K_{k,k}$ cannot have a cycle through all $k$ vertices in one bipartite set which avoids a vertex from the other bipartite set. Hence the graph $K_{3,3}$ shows that a 3-connected cubic graph may not be $C(3,1)$. Similarly, the 3-cube ($Q_3$) demonstrates that a 3-connected plane graph may not be $C(4,1)$. However, the connectivity of a graph does have a strong connection with the property $C(m,n)$.

Matthews and Sumner \cite{MS} conjectured that every 4-connected claw-free graph has cyclability $|V(G)|$, that is, every such graph has a Hamilton cycle. This conjecture still remains open. An immediate corollary of Theorem~\ref{thm:9-vtx-clawfree} guarantees that a 4-connected claw-free graph is $C(9,1)$. But what about 3-connected claw-free graphs? Combining the results of Sections 3,4 and 5, we prove the following theorem which is our main result on cyclability in 3-connected claw-free graphs.\par

\begin{thm}\label{thm:main}
Let $G$ be a 3-connected claw-free graph and $S$ be a set with $k\le 6$ vertices.
Then for every vertex of $S$, $G$ has a cycle avoiding this vertex  but containing all the remaining $k-1$ vertices of $S$. 
\end{thm}

The size of $S$ in Theorem~\ref{thm:main} is the best possible since there are infinitely many 3-connected claw-free graphs which has a vertex subset $S$ of size 7 such that $G$ does not have a cycle avoiding one vertex of $S$, but containing all the remaining $k-1$ vertices. Two main tools that we use to prove Theorem~\ref{thm:main} are Perfect's Theorem and a strengthened version of Perfect's Theorem which is proved in Section 2. 

As a consequence of this result, we can easily guarantee  small topological wheel minors $W_k$ in 3-connected claw-free graphs as follows.

\begin{cor}
Let $G$ be a 3-connected claw-free graph. Then the following hold:\\
(1) $G$ has a  a topological wheel minor $W_k$ with $k\le 5$ if and only if $G$ has a vertex of degree at least $k$;\\
(2) For any vertex $z$ of degree $k\le 5$, $G$ has a topological wheel minor $W_k$ with $z$ as its hub;\\
(3) For any given six vertices, $G$ has a subdivision of $W_3$ (or $K_4$) containing any five of the six vertices on its rim and the remaining vertex as its hub.  
\end{cor}

In section 6, we consider graphs embedded in closed surfaces.  An embedding of a graph in a surface is {\em polyhedral} if the boundaries of every two faces meet properly, i.e., their intersection is either empty, or a single vertex or an edge. An immediate consequence of Theorem~\ref{thm:general-C(3,1)} is the following result for graphs polyhedrally embedded in surfaces.

\begin{thm}\label{thm:poly}
Let $G$ be a graph polyhedrally embedded in a closed surface. Then $G$ has a cycle  through any given three vertices which avoids any other specified vertex. 
\end{thm}

The above result is the best possible since there are infinitely many graphs polyhedrally embedded in a closed surface which have no cycle passing through a certain set of four vertices but avoiding a fifth specified vertex. We describe  such an infinite class in Section 6. Since $K_{3,3}$ has a closed 2-cell embedding in the projective plane, the polyhedral embedding assumption in Theorem~\ref{thm:poly} is necessary. 

\section{Perfect's Theorem and its Generalizations}

We now introduce a theorem on disjoint paths in graphs, due to Perfect \cite{P}, which deserves to be more widely known in graph theory. 
Let $G$ be a graph and  $S$ be a  subset of $V(G)$. For clarity, use``$\backslash$" to denote the deletion operation for sets but use ``$-$" to denote deletion operation for graphs. In other words, $V(G)\backslash S$ is the vertex subset of $V(G)$ with no vertices in $S$, but $G-S$ is the subgraph of $G$ obtained by deleting all vertices in $S$ from $G$ together with all edges incident to the vertices of $S$.

Let $G$ be a graph and $S$ be a non-empty subset of  $V(G)$.
Suppose $v\in V(G)\backslash S$.
We say that $v$ and $S$ are {\it $k$-linked} in $G$ if there exist $k$ internally disjoint paths joining $v$ and $k$ distinct vertices
of $S$ and such that each of the $k$ paths meets $S$ in exactly one vertex.

\begin{thm} [Perfect's Theorem \cite{P}]  \label{thm:PT}
Let $G$ be a graph, and let $x\in V(G)$ and $S\subseteq V(G)\backslash \{x\}$ such that $x$ and $S$ are $k$-linked in $G$.
If $S$ has a subset $T$ of size $k-1$ such that $x$ and $T$ are $(k-1)$-linked,
then there exists a vertex $s\in S-T$ such that $x$ and $T\cup\{s\}$ are $k$-linked.
\end{thm}

The following result strengthens Perfect's Theorem, and will be particularly useful in our work in this paper. Two  disjoint subsets $S_1$ and $S_2$ are {\it $k$-linked} if there exist $k$ disjoint paths joining $k$ distinct vertices of $S_1$ to $k$ distinct vertices of $S_2$ such that each of the $k$ paths meets $S_i$ in exactly one vertex for $i\in \{1,2\}$. 

\begin{thm}\label{thm:strongPT}
Let $G$ be a graph and let $S_1$ and $S_2$ be two disjoint subsets of $V(G)$ such that $S_1$ and $S_2$ are $k$-linked.
If each $S_i$ has a subset $T_i$ of size $k-1$  for $i\in\{1,2\}$  such that $T_1$ and $T_2$ are $(k-1)$-linked,
then there is a vertex $s_i\in S_i-T_i$ for $i\in \{1,2\}$ such that $T_1\cup \{s_1\}$ and $T_2\cup \{s_2\}$ are $k$-linked.  
\end{thm}
\begin{proof}  
Add two new vertices $x_1$ and $x_2$ to the graph $G$, and join $x_i$ to all vertices of $S_i$  for each  $i\in \{1,2\}$. Then identify all vertices in $S_i\backslash T_i$ and denote the resulting vertex by $y_i$. 
Let $G'$ be the resulting graph.
Let $U$ be a minimum vertex-cut of $G'$ which separates $x_1$ and $x_2$.
We claim that: \medskip

\noindent{\bf Claim:} $|U|\ge k$. \medskip

\noindent{\em Proof of Claim.} Suppose to the contrary that $|U|< k$.
If $y_i\in U$ for some $i\in \{1,2\}$, then $U\backslash \{y_i\}$ is a cut of size at most $k-2$ which separates $x_1$ and $x_2$
in $G'-\{y_1,y_2\}$.
So there are at most $k-2$ internally disjoint paths from $x_1$ to $x_2$ in $G'-\{y_1,y_2\}$.
Hence there are at most $k-2$ disjoint paths joining vertices of $T_1$ and vertices of $T_2$, contradicting
the assumption
that $T_1$ and $T_2$ are $(k-1)$-linked.
Hence $U\cap \{y_1,y_2\}=\emptyset$.
It follows that $U$ is a cut separating the subgraphs $x_1y_1$ and $x_2y_2$.  

So $G'$ has at most $k-1$ disjoint paths from $T_1\cup \{y_1\}$ to $T_2\cup\{y_2\}$. On the other hand, since 
$S_1$ and $S_2$ are $k$-linked, there are $k$ disjoint paths between $S_1$ and $S_2$.
So there are $k$ internally disjoint paths between $T_1\cup \{y_1\}$ and $T_2\cup \{y_2\}$, a contradiction.
This contradiction implies that $|U|\ge k$ and the proof of the Claim is complete. \medskip

By the above Claim and Menger's Theorem, the graph $G'$ has $k$ internally disjoint paths joining $x_1$ and $x_2$.
Among these $k$ internally disjoint paths from $x_1$ to $x_2$, one passes through $y_1$ and one passes through $y_2$. Deleting $x_1$ and $x_2$ from these $k$ internally disjoint paths generates $k$ disjoint paths between $T_1\cup \{y_1\}$ and $T_2\cup \{y_2\}$ in $G'$, and hence $k$ disjoint paths between $S_1$ and $S_2$ in $G$.
For $i\in \{1,2\}$, let $s_i\in S_i$ be the endvertex of the path of $G$ corresponding to the path of $G'$ passing through $y_i$.   
Note that $|T_1|=|T_2|=k-1$. So each of these $k$ disjoint paths of $G$ between $S_1$ and $S_2$ meets $S_i$ exactly one vertex. Hence $T_1\cup \{s_1\}$ and $T_2\cup \{s_2\}$ are $k$-linked. This completes the proof.
\end{proof}

The above result cannot be further strengthened 
so as to guarantee that the extended $k$-link contain a $(k-1)$-link between $T_1$ and $T_2$.
For example, the graph in Figure~\ref{fig:link} does not have three disjoint paths between $S_1$ and $S_2$ which contain two disjoint paths between $T_1$ and $T_2$. But $T_1$ and $T_2$ are 2-linked, and $S_1$ and $S_2$ are 3-linked.

\begin{figure}[!htb] 
\begin{center}
\begin{tikzpicture}[thick, scale=1]

\filldraw [] (0,0) circle (3pt); 
\filldraw [] (1.5,0) circle (3pt);
\filldraw [] (1.5,1.5) circle (3pt);
\filldraw [] (1.5,-1.5) circle (3pt);
\filldraw [] (3.5,0) circle (3pt);
\filldraw [] (3.5,1.5) circle (3pt);
\filldraw [] (3.5,-1.5) circle (3pt);
\filldraw [] (5.5,0) circle (3pt);
\filldraw [] (5.5,1.5) circle (3pt);
\filldraw [] (5.5,-1.5) circle (3pt); 
\filldraw [] (7,0) circle (3pt); 

  \draw[] (0,0)--(1.5,0);
   \draw[] (0,0)--(1.5,1.5); 
    \draw[] (0,0)--(1.5,-1.5);

  \draw[thin, black] (1.2, 1.6) to [bend left=50] (1.3, 1.7)--(1.7, 1.7) to [bend left=50] (1.8,1.6)--(1.8, -0.1) to [bend left=50] (1.7, -0.2)--(1.3, -0.2) to [bend left=50] (1.2, -0.1) --(1.2,1.6);
  
  \draw[thin, black] (1.1, 1.7) to [bend left=50] (1.2, 1.8)--(1.8, 1.8) to [bend left=50] (1.9,1.7)--(1.9, -1.7) to [bend left=50] (1.8, -1.8)--(1.2, -1.8) to [bend left=50] (1.1, -1.7) --(1.1,1.7);
  
    \draw[thin, black] (5.2, 1.6) to [bend left=50] (5.3, 1.7)--(5.7, 1.7) to [bend left=50] (5.8,1.6)--(5.8, -0.1) to [bend left=50] (5.7, -0.2)--(5.3, -0.2) to [bend left=50] (5.2, -0.1) --(5.2,1.6);
    
    \draw[thin, black] (5.1, 1.7) to [bend left=50] (5.2, 1.8)--(5.8, 1.8) to [bend left=50] (5.9,1.7)--(5.9, -1.7) to [bend left=50] (5.8, -1.8)--(5.2, -1.8) to [bend left=50] (5.1, -1.7) --(5.1,1.7);
 
 \draw[] (1.5,0)--(3.5,0);
   \draw[] (1.5,0)--(3.5,-1.5);
  \draw[] (1.5,1.5)--(3.5,1.5);
   \draw[] (1.5,1.5)--(3.5,-1.5);
   \draw[] (1.5,-1.5)--(3.5,0);
  \draw[] (1.5,-1.5)--(3.5,1.5);
  
   \draw[] (3.5,0)--(5.5,0);
  \draw[] (3.5,0)--(5.5,1.5);
 \draw[] (3.5,1.5)--(5.5,0);
  \draw[] (3.5,1.5)--(5.5,1.5);
  \draw[] (3.5,-1.5)--(5.5,-1.5);
   \draw[] (3.5,1.5)--(5.5,-1.5);
    \draw[] (3.5,0)--(5.5,-1.5);
    \draw[] (5.5,0)--(7,0);
  \draw[] (5.5,1.5)--(7,0);
   \draw[] (5.5,-1.5)--(7,0);

\node [] at (-.5,0) { $x_1$}; 
 \node [] at (1.5, .7) {$T_1$};
   \node [] at (1.5, -2.2) {$S_1$}; 
   
   \node [] at (7.5,0) { $x_2$}; 
 \node [] at (5.5, .7) {$T_2$};
   \node [] at (5.5, -2.2) {$S_2$};

\end{tikzpicture} 
\caption{\small A 2-link which cannot be extended to a 3-link which contains a 2-link maintaining the initial and terminal vertex sets of the original 2-link.}\label{fig:link}
\end{center}
\end{figure}
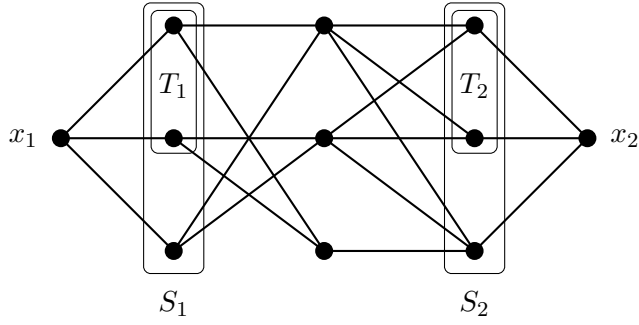

It turns out that Theorem~\ref{thm:strongPT} admits the following `self-refining' version.
Although we will not use it, we include it as it seems to be of some independent interest.

\begin{thm} 
Let $G$ be a graph.
Assume $S_1$ and $S_2$ are two  disjoint subsets of $V(G)$ such that $S_1$ and $S_2$ are $k$-linked.
If each $S_i$ has a subset $T_i$ of size $k-t$  for $i\in\{1,2\}$  such that $T_1$ and $T_2$ are $(k-t)$-linked,
then there is a subset $T_i'\subseteq S_i\backslash T_i$ of size $t$ for $i\in \{1,2\}$ such that $T_1\cup T_1'$ and $T_2\cup T_2'$ are $k$-linked.
\end{thm}

\begin{proof} The proof is by induction on $t$.
If $t=1$, the result follows directly from Theorem~\ref{thm:strongPT}.
So suppose that $t\ge 2$ and that the result holds for all $t'< t$.
Let $T_i$ be a $(k-t)$-subset of $S_i$ such that $T_1$ and $T_2$ are $(k-t)$-linked. 
We need to show that there exists a subset $T_i'\subseteq S_i\backslash T_i$ of size $t$ such that $T_1\cup T_1'$ and $T_2\cup T_2'$ are $k$-linked.

Note that $S_1$ and $S_2$ are $k$-linked and hence also $(k-(t-1))$-linked. 
By Theorem~\ref{thm:strongPT}, there exist vertices $s_i\in S_i\backslash T_i$ for $i=1$ and 2 such that
$T_1\cup \{s_1\}$ and $T_2\cup \{s_2\}$ are $(k-(t-1))$-linked. 
By the induction hypothesis, there are subsets $T_i''\subseteq S_i\backslash (T_i\cup\{s_i\})$ for $i=1$ and 2 of size $t-1$
such that $(T_1\cup \{s_1\})\cup T_1''$ and $(T_2\cup\{s_2\})\cup T_2''$ are $k$-linked. 

Let $T_i'=T''\cup \{s_i\}$.
Then $|T'|=|T''|+1=t$ because $s_i\notin T''_i$. 
Hence $T_1'$ and $T_2'$ are the desired subsets of $S_1\backslash T_1$ and $S_2\backslash T_2$, respectively.  This completes the proof.   
 \end{proof}

\section{Technical Lemmas and Property $C(3,1)$ for Claw-free Graphs} 

Let $G$ be a graph and let $C$ be a cycle of $G$ where we arbitrarily adopt one direction for the traversal of $C$ and call it ``clockwise" and call the opposite direction ``counterclockwise". Suppose $x$ and $y$ are two vertices of $G$. Use $C[x,y]$ to denote the segment of $C$ from $x$ to $y$ in the clockwise direction, and $C^{-1}[x,y]$ to denote the segment of $C$ from $x$ to $y$ in counterclockwise direction. Furthermore, let $C(x,y]$ denote $C[x,y]-x$ and $C(x,y)$ denote $C[x,y]-\{x,y\}$.  For a connected subgraph $Q$ of $G$, the new graph obtained from $G$ by contracting all edges in $Q$ is denoted by $G/Q$.  

The strategy for proving our Theorem~\ref{thm:main} is: first assume that $G$ has a cycle $C$ which contains $k-1$ vertices from a $k$-vertex set $S$,
but avoids a $k^{\mbox{\footnotesize th}}$ vertex in $V(G)\backslash S$, and then apply Perfect's Theorem to the cycle $C$ and the vertex in $S$ which is not in $C$.
Sometimes, we may use Perfect's Theorem more than once in order to find enough paths to help form a suitable cycle containing all $k$
vertices of $S$, but avoiding the given vertex. 

We often encounter the following situation:  $C$ is a cycle, and $P_1$ and $P_2$ are two paths from $x$ and $y$ which end at the same vertex $u$ on $C$.
Then $u$ has degree at least 4, and has two neighbors $u_1$ and $u_2$ on $C$ and a neighbor $u_3$ on $P_1$ and $u_4$ on $P_2$.
Since $G$ is claw-free, either $u_1u_2\in E(G)$ or one of $u_1u_4$ and $u_2u_4$ is an edge of $E(G)$ (see Figure~\ref{fig:jumper}).
Then we use the following ``Jumper Operations": \medskip

(J1) If $u_3u_4\in E(G)$ (dashed edge  in Figure~\ref{fig:jumper}), then  $(P_1\cup P_2-\{u\})\cup \{u_3u_4\}$ is a path joining $x$ and $y$.

(J2) If $u_3u_4 \notin E(G)$, then $G$ must
contain an edge joining a vertex $u_i\in \{u_1,u_2\}$ and a vertex $u_j\in \{u_3,u_4\}$ (dotted edges in Figure~\ref{fig:jumper}).
Otherwise, $G$ has a claw.
If $j=3$, then replace $P_1$ by $P_1'=(P_1-\{u\})\cup \{u_3u_i\}$ which is a path from $x$ to $C$ ending at $u_i\ne u$.
If $j=4$, then replace $P_2$ by $P_2'=(P_2-\{u\})\cup \{u_4u_i\}$ which is a path from $y$ to $C$ ending at $u_i\ne u$.
Then for both $j=3$ 
and
$4$, the graph $G$ has disjoint paths from $x$ and $y$ to $C$ ending at different vertices.  

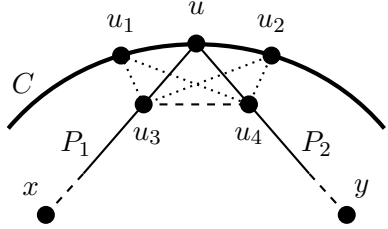
\begin{figure}[!htb] 
\begin{center}
\begin{tikzpicture}[thick, scale=1]
  \coordinate (O) at (0,0);
  \coordinate (A) at (5,0);

  \draw[dashed] (0.5,-1.15) --(1,-.6);

\draw[dashed] (4.5,-1.15)--(4,-.6);

\draw [] (1,-.6)--(2.5, 1.13);
  
  \draw[] (4,-.6)--(2.5, 1.13);
 
  \draw[ultra thick, black] (O) to [bend left=50] (A);
   \draw[thick, dashed] (1.8, .32) -- (3.2,.32);
   \draw[ dotted ] (1.5,.97)--(3.2, .32)--(3.5,.97);
      \draw[ dotted] (1.5,.97)--(1.8, .32)--(3.5,.97);
 
\filldraw [ ] (2.5,1.13) circle (3pt);
\filldraw [ ] (1.5,.97) circle (3pt);
 \filldraw [ ] (3.5,.97) circle (3pt);
 \filldraw [ ] (1.8,0.32) circle (3pt);
 \filldraw [ ] (3.2,0.32) circle (3pt);
\filldraw [] (.5, -1.15) circle (3pt);
\filldraw [] (4.5,-1.15) circle (3pt);

\node [] at (.2,.6) { $C$}; 
\node [] at (0.9,-.2) { $P_1$}; 
\node [] at (4.1,-.2) { $P_2$}; 
\node [] at (2.5,1.6) { $u$}; 

 \node [] at (1.5,1.4) { $u_1$}; \node [] at (3.5,1.4) { $u_2$}; 

\node [] at (1.85,-.1) { $u_3$}; \node [] at (3.2,-.1) { $u_4$}; 

\node [] at (.3, -.8) {$x$};
\node [] at (4.7, -.8) {$y$};

\end{tikzpicture}
\end{center}
\caption{\small Jumpers at the vertex $u$.}\label{fig:jumper}
\end{figure}

By (J1) and (J2), the above circumstance can always be converted to one of the following cases:  

(C1) a cycle $C$ and a path joining $x$ and $y$; or

(C2) a cycle $C$ and two disjoint paths from $x$ and $y$ to $C$ ending at different vertices. \\
So, in the proofs of this paper, we simply say ``by Jumper Operations, we assume that $P_2$ does not end at the endpoint of $P_1$ on $C$".

We now proceed to prove our main theorem by supposing that $G$ is a minimum counterexample with respect to the number of vertices. 
That is, let $G$ be a  
3-connected claw-free graph containing a set $S\subseteq V(G)$ consisting of vertices $x_1,x_2,\ldots,x_k$, where $k\le 5$, and let $z$ be a vertex not in $S$ such that
$G$ has no cycle which contains $S$, but misses $z$.
Moreover, let $G$ be a smallest such graph.\par

\begin{lemma}\label{lem:3-cut}
Let $v$ be a vertex of a minimum counterexample $G$ such that $v\notin S$.
Then $v$ is contained in a 3-cut of $G$.\par
\end{lemma}
\begin{proof}
Since $G$ is claw-free, the subgraph $G-v$ is also claw-free.
If $v$ is not contained in a 3-cut, then $G-v$ is 3-connected and claw-free.
Since $G$ is a minimum counterexample, $G-v$ has a cycle $C$ passing through all vertices of $S$, but avoiding $z$.
But then, the cycle $C$ is a cycle in $G$ of the type we desire, thus contradicting the assumption that $G$ is a counterexample. 
\end{proof}

\begin{lemma}\label{lem:2-component}
Let $G$ be a 3-connected claw-free graph, and $T$ be a 3-cut of $G$. Then $G-T$ has exactly two components $Q_1$ and $Q_2$ such that
each $Q_i$ does not contain a cutvertex.
\end{lemma} 
\begin{proof}First, we prove that $G-T$ has precisely two components. 
Assume to the contrary that $G-T$ has at least three components, say $Q_1, Q_2$ and $Q_3$.
Since $G$ is 3-connected, every vertex of $T$ has a neighbor in each of these components.
So this vertex of $T$, together with these neighbors, induce a claw, contradicting the assumption that $G$ is claw-free.
So $G-T$ has exactly two connected components $Q_1$ and $Q_2$ as claimed.  

In the following, we show that each $Q_i$  does not contain a cutvertex.
Assume, to the contrary, that $v$ is a cutvertex of $Q_i$.
Then $Q_i$ has exactly two blocks $B_1$ and $B_2$ (i.e., maximal 2-connected subgraphs) separated by $v$ since $G$ is claw-free.
Note that, for any vertex $v_j\in T$, the vertex $v_j$ has neighbors in both $Q_1$ and $Q_2$ since $G$ is 3-connected. It follows that
the neighbors of $v_j$ in $Q_i$ induce a clique because $G$ is claw-free. So all neighbors of $v_j$ in $Q_i$ belong to the same block
of $Q_i$. Let 
\[ U_t=\{v_j | v_j\in T \mbox{ and }\big (V(B_t)\cap (N(v_j)\big )\backslash v)\ne \emptyset \} \mbox{ for } t\in \{1,2\},\] where
$N(v_j)$ is the set of all neighbors of $v_j$ in $G$. Then $U_1\cap U_2=\emptyset$ and $|U_1|+|U_2|=|T|=3$.  
Without loss of generality, assume  $ |U_1|\ge |U_2|$.  
Then $U_2\cup \{v\}$ is a vertex cut of $G$ separating $B_2$ from the remaining subgraph of $G$ and has size at most two, which contradicts the assumption that $G$ is 3-connected. 
This completes the proof.
\end{proof}

\begin{lemma}\label{lem:z-3-cut}
Let $T=\{z, w_1,w_2\}$ be a 3-cut of a minimum counterexample $G$. Then $G-T$ has a component which is a singleton vertex from $S$. 
\end{lemma}
\begin{proof}
By Lemma~\ref{lem:2-component}, $G-T$ has two components $Q_1$ and $Q_2$.
If $Q_i\cap S=\emptyset$, let $G'=(G-V(Q_1))\cup\{zw_1w_2z\}$ where $zw_1w_2z$ is a 3-cycle. 
Then $G'$ is claw-free.
By Lemma~\ref{lem:2-component}, $G'$ is 3-connected.
So $G'$ has a cycle $C$ passing through all the vertices in $S$, but avoiding $z$.
If $C$ contains $w_1w_2$ and $w_1w_2$ is not an edge of $G$, use a path $P$ of $G-V(Q_j)$ ($j\ne i$) joining $w_1$ and $w_2$ to replace the edge
$w_1w_2$ in $C$ to get a cycle $C'$ in $G$ of the type desired.
Therefore, both $Q_1\cap S\ne \emptyset$ and $Q_2\cap S\ne \emptyset$ hold. 

If $|V(Q_i)|\ge 2$ for both $i=1$ and 2, consider $G_i=(G/Q_i)\cup \{zw_2, zw_1\}$.
Note that $G_i$ is 3-connected and claw-free.
Since $G$ is a minimum counterexample, both $G_1$ and $G_2$ are not counterexamples.
Let $q_i$ be the vertex of $G_i$ obtained by contracting $Q_i$.
Let $S_i=(S\backslash V(Q_i)) \cup \{q_i\}$.
Then $|S_i|\le |S|$, and hence $G_i$ has a cycle $C_i$ containing all vertices of $S_i$, but avoiding $z$.
Then $(C_1-q_1)\cup (C_2-q_2)$ is a cycle of $G$ containing all vertices of $S$, but avoiding $z$, and this contradicts
the assumption that $G$ is a counterexample.
Therefore, one of the two components is a singleton vertex from $S$. 
\end{proof}


\begin{thm}  
Let $G$ be a 3-connected claw-free graph. Then $G$ has the property $C(3,1)$.
\end{thm}
\begin{proof} The proof is by contradiction.
So suppose $G$ is a minimum counterexample; i.e., suppose there exist three vertices $x_1,x_2,x_3$ and a fourth vertex $z$ in $G$ such that there is no cycle in $G$ containing $S=\{x_1,x_2,x_3\}$ which avoids $z$. 

By Lemmas~\ref{lem:3-cut} and \ref{lem:z-3-cut}, $z$ belongs to a 3-cut $T=\{z,v_1,v_2\}$ such that $G-T$ consists of precisely two components, one of which is a singleton from $S$, say $x_1$.
Since $G-z$ is 2-connected, $x_1$ is contained in a cycle which must contain $v_1$ and $v_2$.
If $\{v_1,v_2\}=\{x_2,x_3\}$, then any cycle of $G-z$ containing $x_1$ is a cycle of the type sought, contradicting the assumption
that $G$ is a counterexample.
So assume that $x_2\notin \{v_1,v_2\}$.  

By Menger's Theorem, there are three internally disjoint paths from $x_2$ to three distinct vertices of $T$, two of which end at $v_1$ and $v_2$ respectively.
These two paths, together with $v_2x_1v_1$, form a cycle $C$ containing both $x_1$ and $x_2$, but not $z$.
Denote the third path from $x_2$ to $z$ by $P'$.
If $x_3\in V(C)$, then $C$ is a cycle of the type desired, a contradiction again.
So in the following, assume that $x_3\notin  V(C)$.
Then by Menger's Theorem, there are three internally disjoint paths $P_1, P_2$ and $P_3$ from $x_3$ to three distinct vertices of $V(C)$. 

If none of the paths $P_1, P_2$ and $P_3$ contains $z$, then all three end in the segment $C[v_1,v_2]$.
Then one of the segments $C[v_1,x_2]$ and $C[x_2,v_2]$ contains two of the endvertices of $P_1,P_2$ and $P_3$.
Without loss of generality, assume $C[v_1,x_2]$ contains the endvertices of $P_1$ and $P_2$, denoted by $u_1$ and $u_2$, which appear in clockwise order on $C$, respectively. Then $C[u_2,u_1]$ contains both $x_1$ and $x_2$. Hence $C'=P_1\cup P_2\cup C[u_2,u_1]$ is a cycle containing all vertices of $S$, but not $z$, again a contradiction. 
 So, in the following, suppose that 
\begin{description}
\item [ ] ($*$) {\sl among any three internally disjoint paths $P_1, P_2$ and $P_3$  from $x_3$ to three distinct vertices of $C$, there is always one containing $z$.}
\end{description}
Without loss of generality, assume that $P_3$ contains $z$. Let $P''$ be the subpath of $P_3$ joining $x_3$ and $z$. Let $w$ be the first vertex in $P''\cap P'$ encountered when traversing $P''$ from $x_3$ to $z$.

If $w\ne z$, let $P_3'$ be the path obtained by traversing $P''$ from $x_3$ to $w$ and then traversing $P'$ from $w$ to $x_2$. 
If $w=z$, then there is an edge $e$ joining the vertex of $P''$ in  $N(z)\backslash (T\cup x_1)$ and the vertex of $P'$ in $N(z)\backslash (T\cup x_1)$ since $G$ is claw-free and hence the  neighbors of $z$ induce a clique in $G-(T\cup \{x_1\})$.
Then let $P_3'=\big ( (P''\cup P')-z\big )\cup e$. So, no matter whether $w=z$ or $w\ne z$, $P_3'$ is a path from $x_3$ to $x_2$ which avoids $z$.  
By Jumper Operations, we may assume that $P_1, P_2$ and $P_3'$ end at different vertices of $C$.  

By ($*$),  it follows that $(P_1\cup P_2)\cap P_3' \ne \emptyset$. Since $P_1$ and $P_2$ are internally disjoint from $P_3$ and $P_3$ contains $P''$ as a subpath, it follows that $(P_1\cup P_2)\cap P'\ne \emptyset$. Let $w''$ be the first vertex in $(P_1\cup P_2)\cap P'$ encountered when traversing $P'$ from $x_2$ to $z$. Without loss of generality, assume $w''\in P_1$. Let $P_1'$ be the path from $x_3$ to $C$ obtained by traversing  $P_1$ from $x_3$ to $w''$ and then along $P'$ to $x_2$. Because of the choice of $w''$,  $P_1'$ and $P_2$ are two internally disjoint paths from $x_3$ to $C$ ending at two distinct vertices in the same closed segment of $C$
determined by $x_1$ and $x_2$. Then $x_3$ can be inserted into $C$ by using $P_1'$ and $P_2$ to generate a new cycle $C'$ containing all the vertices of $S$, but not $z$, where $C'=C[x_1,x_2]\cup P_1'\cup P_2\cup C[u_2,x_1]$, and this again contradicts $G$ being a counterexample.
This completes the proof.  
\end{proof}

In the rest of this section, we derive more properties of 3-cuts of a minimum counterexamples to $C(k,1)$ with $k\le 5$.

\begin{lemma}\label{lem:3cuts-intersection}
Let $G$ be a minimum counterexample to $C(k,1)$ with $k\le 5$ and let $T$ and $T'$ be two distinct 3-cuts of $G$ such that $|T\cap T'|= 2$.
Then  $G-(T\cup T')$ has two components. Furthermore, if $z\in T\cap T'$, then $G-z$ has a Hamiltonian cycle. 
\end{lemma}
\begin{proof}
Assume that $T\cap T'=\{v_1,v_2\}$, $T\backslash (T\cap T')=\{u\}$ and $T'\backslash (T\cap T')=\{u'\}$. 

Since $T$ is a 3-cut, one of the components of $G-T$ does not contain $u'$ which we will denote by $Q_1$. Note that $Q_1$ is also a component of $G-(T\cup T')$.
Similarly, the component $Q_2$ of $G-T'$ not containing $u$ is also a component of $G-(T\cup T')$.
Since $G$ is 3-connected, both $v_1$ and $v_2$ have neighbors in $Q_1$ and $Q_2$.
It follows, since $G$ is claw-free, that both $v_1$ and $v_2$ have no neighbors in components of $G-(T\cup T')$ other than $Q_1$ and $Q_2$. If $G-(T\cup T')$ has a component different from $Q_1$ and $Q_2$, then the component is separated by $\{u,u'\}$ from the remaining subgraph, which contradicts the assumption that $G$ is 3-connected. This contradiction implies that  $G-(T\cup T')$ has exactly two components $Q_1$ and $Q_2$.  

Now, assume that $z\in T\cap T'$ and, without loss of generality, assume $z=v_1$. By  Lemma~\ref{lem:z-3-cut}, both $Q_1$ and $Q_2$ are single vertices from $S$, denoted by $x$ and $y$. Since $G$ is 3-connected, there is an edge joining $u$ and $u'$. But
then $xv_2yu'ux$ is a Hamiltonian cycle of $G-z$. 
\end{proof}

\begin{lemma}\label{lem:avoid-z-path}
Let $G$ be a minimum counterexample such that $G$ has exactly one 3-cut $T$ containing $z$, and $T$ separates a single vertex $x_1\in S$ from $Q=G-(T\cup \{x_1\})$. Then for some vertex $x\in V(Q)$ and a connected subgraph $H$ of $G-z$ with at least three vertices in $G-\{z,x_1\}$,
the vertex $x$ and $H$ are 3-linked in $G-z$. 
\end{lemma}
\begin{proof} Let $H$ be a connected subgraph of $G-z$ containing at least three vertices of $G-\{z,x_1\}$.
Suppose to the contrary that $G-z$ does not contain three internally disjoint paths joining $x$ to three distinct vertices of $H$.
In other words, $G-z$ has a 2-cut separating $x$ and $H$ by Menger's Theorem.
This 2-cut, together with $z$, forms a 3-cut of $G$ which we shall denote by $T'$. 
By Lemma~\ref{lem:z-3-cut}, $T'$ separates a single vertex of $S$ from the remaining subgraph of $G$.
This single vertex must be $x$ since $H$ has more than one vertex. Note that $x\in V(Q)=V(G)\backslash (T\cup \{x_1\})$.
So $T\ne T'$, a contradiction to the assumption that $G$ has exactly one 3-cut $T$ containing $z$.  
\end{proof}

\section{Property $C(4,1)$ for Claw-free Graphs}

In this section, we are going to show that every 3-connected claw-free graph is $C(4,1)$.

\begin{lemma}\label{lem:c41-z-in-2cut}
Let $G$ be a minimum counterexample to the $C(4,1)$ property.
That is, suppose $G$ is a smallest 3-connected claw-free graph containing a vertex $z$ and four additional vertices $x_1,x_2,x_3$ and $x_4$ such that $G$ has no cycle containing all vertices $x_i$'s, but avoiding $z$.
Then $z$ is contained in exactly one 3-cut in $G$. 
\end{lemma}
\begin{proof}
Let $G$ be as stated above and let $S=\{x_1,x_2,x_3,x_4\}$.
Suppose to the contrary that $z$ is contained in at least two 3-cuts.
By Lemma~\ref{lem:z-3-cut} and the fact that $G$ is claw-free, $z$ is not contained in three different 3-cuts.
So assume that $z$ is contained in exactly two 3-cuts $T$ and $T'$.
By Lemma~\ref{lem:3cuts-intersection}, we have $T\cap T'=\{z\}$.

 Let $x_1\in S$ be separated by $T$ and $x_2\in S$ be separated by $T'$.
Denote $G-(T\cup T'\cup \{x_1,x_2\})$ by $Q$.  If $Q$ has at least two components $Q_1$ and $Q_2$, then at least one vertex in $(T\cup T')\backslash \{z\}$ has neighbors in both $Q_1$ and $Q_2$ since $G$ is 3-connected. This vertex together with its neighbors form a claw, a contradiction. Hence $Q$ is connected.

Since $G$ is claw-free, $z$ has no neighbor in $Q$.  Suppose $Q$ contains at most one vertex $x_3$ from $S$. Then $x_4$ belongs to either $T$ or $T'$. Then $G$ has a cycle $C$ containing $x_1, x_2$ and $x_3$, but avoiding $z$.
Note that the cycle $C$ must contain all the vertices of $(T\cup T')\backslash \{z\}$.
Then $C$ contains all the vertices of $S$, a contradiction to the assumption that $G$ is a counterexample.

So in the following, assume that $Q$ contains both $x_3$ and $x_4$.
Since $G$ has the $C(3,1)$-property, $G$ has a cycle $C$ containing $x_1, x_2, x_3$, but avoiding $z$.
Without loss of generality, assume that $x_3\in C(x_2, x_1)$.
By Menger's Theorem, $G$ has three internally disjoint paths  $P_1, P_2$ and $P_3$  joining $x_4$ and three distinct vertices of $C$ which do not contain $z$. If two of these three paths end on  the same closed segment of $C$ determined by $x_1, x_2$ and $x_3$, then we can insert $x_4$ into the cycle $C$ to obtain a cycle containing all the vertices of $S$, but not containing $z$, a contradiction to the assumption that $G$ is a counterexample. 
Hence, we may assume the three paths $P_1, P_2$ and $P_3$ end at $u_1, u_2$ and $u_3$, respectively, such that $u_i\in C(x_i, x_{i+1})$
for $i=1,2,$ and $u_3\in C(x_3, x_1)$.

By applying Perfect's Theorem to $x_3$ and the subgraph $H=C[u_3,u_2]\cup P_2\cup P_3 \cup P_1$, there are three internally disjoint paths $P_1', P_2'$ and $P_3'$ 
joining $x_3$
and three distinct vertices of $H$ which do not contain $z$ such that $P_2'$ and $P_3'$ end at $u_2$ and $u_3$, respectively.
Let $w$ be the vertex at which $P_1'$ ends.
Note that $w\notin P_2\cup P_3\cup C[x_2, u_2]\cup C[u_3,x_1]$.
For otherwise, one could insert $x_3$ into the cycle $C[u_3,u_2]\cup P_2\cup P_3$ by the path $P_1'$ and one of $P_2'$ and $P_3'$ to form a cycle containing all the vertices from $S$, but avoiding $z$.
Therefore, we may assume that either $w\in P_1$ or $w\in C(x_1, x_2)$.

Note that, if $w=u_1$, we could use Jumper Operations to modify the path $P_1$ or the cycle $C$ at $u_1$ to reduce the case $w=u_1$ to $w\ne u_1$.  

If $w\in P_1-u_1$, then there are two disjoint paths from $x_4$ to the segment of $C[x_2, x_3]$ (one is $P_2$ and the other is from $x_4$ to $w$ along $P_1$ then to $x_3$ along $P_1'$), which could be used to insert $x_4$ into the cycle $C$ to get a cycle through all vertices in $S$, but avoiding $z$, a contradiction.

So, in the following, assume that $w\in C(x_1,u_1)$ or $w\in C(u_1, x_2)$. By symmetry, it suffices to consider  $w\in C(x_1, u_1)$.  Then
$C'=C[x_1, w]\cup P_1'\cup P_2'\cup C^{-1}[u_2, u_1]\cup P_1\cup P_3\cup C[u_3, x_1]$
is a cycle of the type we seek, a contradiction to the assumption that $G$ is a counterexample.
Hence $z$ is contained in exactly one 3-cut and the proof is complete.
\end{proof}

\begin{thm}\label{thm:c41}
Let $G$ be a 3-connected claw-free graph.
Then $G$ satisfies property $C(4,1)$. 
\end{thm}

\begin{proof}
Let $G$ be a minimum counterexample. 
By Lemma~\ref{lem:c41-z-in-2cut}, $z$ is contained in exactly one 3-cut $T$.
By Lemma~\ref{lem:z-3-cut}, $T$ separates a single vertex from $S$, say $x_1$. Let $Q$ be a component of $G-T$ not containing $x_1$. 
\medskip

\noindent{\bf Claim.} {\sl $T\cap S=\emptyset$.}
\medskip

\noindent{\em Proof of the Claim.} If not, assume that $|T\cap S|\ge 1$.
Then $Q\cap S$ consists of at most two vertices, say $x_2$ and $x_3$.
So $x_4\in T$.
Since $G$ satisfies property $C(3,1)$, $G$ has a cycle containing $x_1, x_2$ and $x_3$, but avoiding $z$.
Since $T$ is a 3-cut separating $x_1$, it follows that both vertices of $T\backslash \{z\}$ are contained in this cycle, which implies that the cycle contains all the vertices of $S$, a contradiction of the assumption that $G$ is a counterexample.
Hence $T\cap S=\emptyset$ as claimed. \medskip

By the claim, $S\backslash \{x_1\}\subseteq V(Q)$.
By the property $C(3,1)$, let $C$ be a cycle of $G$ containing $x_1, x_2$ and $x_3$, but avoiding $z$.
Since $G$ is a counterexample, $x_4$ is not on $C$. Since $z$ is contained in exactly one 3-cut, by Lemma~\ref{lem:avoid-z-path}, $x_4$ and 
$C$ are 3-linked in $G-z$. Therefore, there are three internally disjoint paths $P_1,P_2$ and $P_3$ joining $x_4$ and three distinct vertices of $C$ each of which avoids $z$.





Assume that $P_1, P_2$ and $P_3$ end on $C$ at $u_1, u_2$ and $u_3$, respectively.
Note that no two of these paths end in the same closed segment of $C$ determined by $x_1, x_2$ and $x_3$, or otherwise we could insert $x_4$ into $C$ using the two paths to get a cycle of the type desired, a contradiction to the assumption that $G$ is a counterexample.
Without loss of generality, assume that $u_i\in C(x_i, x_{i+1})$ for each $i\in \{1,2\}$ and $u_3\in C(x_3, x_1)$. 
Let $H=C[x_1, u_1]\cup P_1\cup P_2\cup C[u_2, x_1]\cup P_3$.
By Lemma~\ref{lem:avoid-z-path},  $x_2$ and $H$ are 3-linked in $G-z$.
Applying Perfect's Theorem to $x_2$ and $H$ in $G-z$, there are three disjoint internally disjoint paths $P_1', P_2'$ and $P_3'$ joining $x_2$ and three distinct vertices of $H$ such that $P_1'$ and $P_2'$ end at $u_1$ and $u_2$, respectively.
Assume that $P_3'$ ends at $w$.
Then $w$ does not belong to $C[x_1, u_1] \cup P_1\cup P_2\cup C[u_2,x_3]$.
Otherwise, we could insert $x_2$ into the cycle $C[x_1,u_1]\cup P_1\cup P_2\cup C[u_2,x_1]$ which already contains $x_1, x_4$ and $x_3$ to generate a cycle of the type desired, a contradiction to the assumption that $G$ is a counterexample. 

Note that the case $w=u_3$ can
be
converted to the case $w\ne u_3$ by using Jumper Operations at the vertex $w=u_3$.
So it suffices to consider the cases when $w\in P_3-\{x_4, u_3\}$, when $w\in C(u_3, x_1)$ or when $w\in C(x_3, u_3)$. 

If $w\in P_3-\{x_4, u_3\}$, then $x_4$ can be inserted into the cycle $C[u_2,u_1]\cup P_1'\cup P_2'$ by replacing $P_1'$ with the path from $u_1$ to $x_4$ along $P_1$ and to $w$ along $P_3$ and then to $x_2$ along $P_3'$.
Therefore, $G$ has a cycle containing all the vertices of $S$, but avoiding $z$, again a contradiction. 

If $w\in C(u_3, x_1)$, then $C'=C[x_1, u_1]\cup P_1\cup P_3\cup C^{-1}[u_3, u_2]\cup P_2'\cup P_3'\cup C[w,x_1]$ is a cycle containing all the vertices of $S$, but avoiding $z$, a contradiction.
If $w\in C(x_3, u_3)$, then $C'=C[x_1,x_2]\cup P_3'\cup C^{-1}[w,u_2]\cup P_2 \cup P_3\cup C[u_3, x_1]$ is a cycle containing all the vertices of $S$, but avoiding $z$, a contradiction yet again.
This completes the proof of the theorem. 
\end{proof}

\section{Property $C(5,1)$ for Claw-free Graphs}

We are now prepared to prove the last case $C(5,1)$ of our main result on claw-free graphs. As we shall see, Theorem~\ref{thm:strongPT} will play an important role in the our proof.

\begin{lemma}\label{C51-z-in-2cut}
Let $G$ be a minimum counterexample to $C(5,1)$.
That is, suppose $G$ is a smallest 3-connected claw-free graph containing a vertex $z$ and five additional vertices $x_1,x_2,x_3, x_4$ and $x_5$ such that $G$ has no cycle containing all the $x_i$'s, but avoiding $z$. Then $z$ belongs to exactly one 3-cut. 
\end{lemma}

\begin{proof} Let $G$ be as stated above and let $S=\{x_1,x_2,x_3,x_4,x_5\}$.
Suppose to the contrary that $z$ belongs to at least two different 3-cuts, say $T$ and $T'$. 
By Lemma~\ref{lem:z-3-cut} and the fact that $G$ is claw-free, $T$ and $T'$ are the only two 3-cuts of $G$ containing $z$. Without loss of generality, assume that $T$ separates the single vertex $x_1\in S$ and $T'$ separates the single vertex $x_2\in S$ by Lemma~\ref{lem:z-3-cut}. Let $Q=G-(T\cup T'\cup \{x_1,x_2\})$. An argument similar to that used in the proof of Lemma~\ref{lem:c41-z-in-2cut} shows that $Q$ is connected and $z$ has no neighbors in $Q$. 

By Theorem~\ref{thm:c41}, $G$ has a cycle $C$ containing $x_1, x_2$ and two other vertices from $S$, say $x_3$ and $x_4$, but avoiding $z$.  
\medskip

\noindent{\bf Claim~1.} {\sl Either $C(x_1, x_2)\cap S=\emptyset$ or  $C(x_2,x_1)\cap S=\emptyset$.} 
\medskip

\noindent{\em Proof of Claim~1:} Suppose the Claim is false and assume that $x_3\in C(x_1,x_2)\cap S$ and $x_4\in C(x_2, x_1)\cap S$.
By Menger's Theorem, there are three internally disjoint paths $P_1, P_2$ and $P_3$ joining $x_5$ and three distinct vertices of $C$, say $u_1, u_2$ and $u_3$,
respectively. Note that $P_1,P_2$ and $P_3$ do not contain $z$ because they are paths in $Q$. 
Note that $x_1, x_2, x_3$ and $x_4$ separate $C$ into four segments, none of which contains two vertices of $\{u_1,u_2, u_3\}$.
For otherwise, one could use the two paths with endvertices in the same segment to insert $x_5$ into $C$ to generate a cycle
of the type we seek.
It follows that: 

\begin{description}
\item[($**$) ] {\sl  any three internally disjoint paths from $x_5$ to $C$ end in three different segments of $C$ determined by
$x_1,\ldots, x_4$.}
\end{description}

By symmetry, we may assume that $u_1\in C(x_1, x_3]$, $u_2\in C(x_2, x_4)$ and $u_3\in C(x_4, x_1)$.
Now, applying Perfect's Theorem to $x_4$ and $H=C[u_3,u_2]
\cup (\bigcup P_i)$, we obtain three internally disjoint paths joining $x_4$ and three distinct vertices of $H$, and two of them, say $P_1'$ and $P_2'$, end at $u_2$ and $u_3$.
Assume the third path $P_3'$ from $x_4$ to $H$ ends at $w$. Note that $P_1',P_2'$ and $P_3'$ do not contain $z$.
By $(**)$ and Jumper Operations, we may assume that $w\notin P_1$. 
The vertex $w$ does not belong to $P_2\cup P_3\cup C[u_3,x_1)\cup C(x_2, u_2]$, for otherwise $x_4$ could be inserted into
the cycle $C[u_3,u_2]\cup P_2\cup P_3$ by two of the three paths from $P_1', P_2'$ and $P_3'$ to generate the cycle sought
and again we would have a contradiction.
Hence, $w\in C(x_1, u_1)$ or $ C(u_1, x_3]$ or $C(x_3,x_2)$. 

If $w\in C(x_1, u_1)$, then $C'=C[x_1, w]\cup P_3'\cup P_1'\cup C^{-1}[u_2, u_1]\cup P_1\cup P_3 \cup C[u_3, x_1]$ is a
cycle of the type desired and again we have a contradiction.
If $w\in C(u_1, x_3]$,
then cycle $C'= C[x_1, u_1]\cup P_1\cup P_2\cup C^{-1}[u_2, w]\cup P_3'\cup P_2' \cup C[u_3, x_1]$ again yields a contradiction. 
So suppose that $w\in C(x_3,x_2)$. Then applying Perfect's Theorem to $x_3$ and $H=C[u_3,u_1]\cup C[w, u_2] \cup (\bigcup P_i)\cup (\bigcup P_i')$, we obtain three internally disjoint paths $P_1'', P_2''$ and $P_3''$ from $x_3$ to $H$ ending at $u_1, w$ and a third vertex $w'$, respectively. By Jumper Operations, $w'\notin\{u_2, u_3, x_4, x_5\}$.   Again, none of these three paths contains $z$ because they are in $Q$.  Note that $w'\notin C(x_1,u_1)\cup C(w, x_2)\cup P_1\cup P_3'$. Otherwise, $x_3$ can be inserted into the cycle $C[u_3,u_1]\cup P_1\cup P_2\cup C^{-1}[u_2,w]\cup P_3'\cup P_2'$ by using two paths from among the $P_i''$'s to generate a cycle of the type desired, a contradiction. Similarly, $w'\notin P_2\cup P_3$ (otherwise, $x_5$ could be inserted to the cycle $C[u_3,u_1]\cup P_1''\cup P_2''\cup C[w, u_1]\cup P_1'\cup P_2'$ to yield a contradiction) and $w'\notin P_1'\cup P_2'$ (otherwise, $x_4$ could be inserted to the cycle $ C[u_3,u_1]\cup P_1''\cup P_2''\cup C[w, u_1]\cup P_2\cup P_3$ to yield a contradiction). Therefore, $w'\in C(x_2, u_2)$ or $C(u_3, x_1)$. If $w'\in C(x_2,u_2)$, then $C'=C[x_1, u_1]\cup P_1''\cup P_3''\cup C^{-1}[w', w]\cup P_3'\cup P_1'\cup P_2\cup P_3\cup C[u_3,x_1]$ is a cycle of the type desired, a contradiction. If $w'\in C(u_3, x_1)$, then $C'=C[x_1,u_1]\cup P_1\cup P_3\cup P_2'\cup P_1'\cup C^{-1}[u_2, w]\cup P_2''\cup P_3''\cup C[w',x_1]$ is a cycle of the type desired, a contradiction again. 
This completes the proof of Claim~1.

\medskip

By Claim~1, every cycle $C$ containing $x_1, x_2$ and two other vertices from $S$ satisfies the property that
one of $C(x_1,x_2)$ and $C(x_2,x_1)$ does not contain vertices from $S$.
Without loss of generality, assume that $C(x_1,x_2)$ fails to contain vertices from $S$.
So then $C(x_2,x_1)$ must contain two vertices from $S$, say  $ x_3$ and $x_4$. Assume that $x_1, x_2, x_3$ and $x_4$ appear clockwise around the cycle $C$.
By Menger's Theorem, there are three disjoint paths $P_1, P_2$ and $P_3$ from $x_5$ to  $C$  ending at three distinct vertices $u_1, u_2$ and $u_3$, respectively. Since all three paths belong to $Q$, they do not contain the vertex $z$. 
Moreover, each of the segments $C[x_1,x_2]$, $C[x_2,x_3]$, $C[x_3,x_4]$ and $C[x_4,x_1]$ contains at most one vertex from $\{u_1,u_2,u_3\}$. \medskip

\noindent{\bf Claim~2.} {\sl The segment $C[x_3,x_4]$ does not contain any vertex from $\{u_1,u_2, u_3\}$.} \medskip

\noindent{\em Proof of Claim~2.} Assume, to the contrary, that $C[x_3,x_4]$ does contain a vertex from $\{u_1,u_2,u_3\}$. 

First,  assume that $C[x_1,x_2]$ does not contain a vertex of $\{u_1,u_2,u_3\}$.
Without loss of generality, we further assume that $u_1\in C(x_2,x_3)$, $u_2\in C(x_3,x_4)$ and
$u_3\in C(x_4,x_1)$, respectively.
Now apply Perfect's Theorem to $x_3$ and $H=C[x_1,u_1]\cup P_1\cup P_2 \cup C[u_2, x_1]\cup P_3$. There are three internally disjoint paths joining $x_3$ and three distinct vertices of $H$ such that two of them, say $P_1'$ and $P_2'$, end at $u_1$ and $u_2$, and the third path  $P_3'$ ends at some vertex $w$. Using Jumper Operations, we may assume that $w\notin \{u_3, x_5\}$. 
A routine check shows that $w$  must belong to $C(x_1,x_2)$, 
for otherwise, there exists a cycle through $x_1,...,x_5$ which avoids $z$, yielding a contradiction.

Now apply Perfect's Theorem to $x_4$ and $H'=C[u_3, u_1]\cup (\bigcup P_i) \cup (\bigcup P_i')$. There are three internally
disjoint paths  joining $x_4$ and three distinct vertices of $H'$ such that two of them, say $P_1''$ and $P_2''$, end at $u_2$ and $u_3$ and the third path $P_3''$ ends at some vertex $w'$. By Jumper Operations, $w'\notin\{x_3, x_5, u_1, w\}$. 
By symmetry of $x_3$ and $x_4$, the vertex $w'$ must belong to $C(x_1,x_2)$ or $P_3'$.
If $w'\in P_3'$, 
then $C'=C[x_1,u_1]\cup P_1'\cup P_3'[x_3,w']\cup P''_3\cup P_1''\cup P_2\cup P_3\cup C[u_3,x_1]$ is a cycle of the type we need, a contradiction.
Therefore, $w'\in C(x_1,x_2)$. Now let $C'=C[x_1,w]\cup P'_3\cup P_2' \cup P_1'' \cup P_3''\cup C[w',u_1]\cup P_1\cup P_3\cup C[u_3, x_1]$ if $w'\in C(w,x_2)$,  or let $C'=C[x_1,w']\cup P''_3\cup P_1'' \cup P_2' \cup P_3'\cup C[w,u_1]\cup P_1\cup P_3\cup C[u_3,x_1]$ if  $w'\in C(x_1,w)$. Then $C'$ is a cycle containing all five vertices
from $S$, but avoiding $z$, a contradiction. This contradiction implies that $C[x_1,x_2]$ does contain a vertex of $\{u_1,u_2,u_3\}$.

So in the following, assume that $u_1\in C(x_1,x_2)$ and $u_2\in C[x_3,x_4]$. Then $u_3$ belongs to either $C(x_2, x_3)$ or $C(x_4, x_1)$. By symmetry, without loss of generality, let us assume that $u_3\in (x_4, x_1)$. Now, apply Perfect's Theorem to $x_4$ and $H=C[u_3, u_2]\cup (\bigcup P_i)$. 
There are three internally disjoint paths all avoiding $z$ joining $x_3$ and three distinct vertices of $H$ such that two of them, say $P_1'$ and $P_2'$, end at $u_2$ and $u_3$, and the third path $P_3'$ end at some vertex $w$. By Jumper Operations, $w\notin \{x_5,u_1\}$. Note that $w$ does not belong to $P_2, P_3, C(x_3, u_2)$ or $C(u_3, x_1)$ because, otherwise, $x_4$ could be inserted into the cycle $C[u_3, u_2]\cup P_2\cup P_3$ by using two of the three paths $P_1',P_2'$ and $P_3'$ to generate a cycle of the type desired, a contradiction. If $w\in P_1$, then $x_5$ can be inserted into the cycle $C[u_3, u_2]\cup P_1'\cup P_2'$ by the paths $P_2, P_1$ and $P_3'$ to generated a cycle of the type desired, a contradiction again. If $w\in C(x_1, u_1)$, then $C'=C[x_1,w] \cup P_3'\cup P_1'\cup C^{-1}[u_2, u_1]\cup P_1\cup P_3\cup C[u_3, x_1]$ yields a cycle of the type sough, a contradiction. If $w\in C(u_1, x_2)$, then $C'=C[x_1, u_1]\cup P_1\cup P_2\cup C^{-1}[u_2, w]\cup P_3'\cup P_2'\cup C[u_3, x_1]$ yields a cycle the type we want, a contradiction again. So $w\in C(x_2,x_3)$. 

Now apply Perfect's Theorem to $x_3$ and $H''=C[u_3, w]\cup (\bigcup P_i)\cup (\bigcup P_i')$. There are three internally disjoint paths avoiding $z$ which join $x_3$ and three distinct vertices of $H''$ such that two of them, say $P_1''$ and $P_2''$, end at $w$ and $u_2$, respectively, and the third path $P_3''$ ends at some vertex $w'\notin \{u_1, u_3, x_4, x_5\}$ by Jumper Operations. Note that $w'\notin C[x_2, w] \cup P_3' \cup P_2\cup P_1'$ since, otherwise, $x_3$ could be inserted into $C[u_3, w]\cup P_3'\cup P_1'\cup P_2\cup P_3$ to generate a cycle of the type desired, a contradiction. Similarly, $w'\notin P_1\cup P_3$ (otherwise, $x_5$ could be inserted into $C[u_3, w]\cup P_1''\cup P_2''\cup P_1'\cup P_2'$) and $w'\notin P_2'$ (otherwise, $x_4$ could be inserted into $C[u_3, w]\cup P_1''\cup P_2''\cup P_2\cup P_3$ by $P_3', P_3''$ and a subpath of $P_2'$).  
If $w'\in C(x_1,x_2)$, then replace the subpath of $C(x_1, x_2)$ joining $w'$ and $u_1$ by $P_1\cup P_2\cup P_2''\cup P_3''$ (which contains both $x_3$ and $x_5$) in the cycle $C[u_3, w]\cup P_3'\cup P_2'$ to generate a cycle that we need, a contradiction. So $w\in C(u_3, x_1)$. However, the cycle $C'= C[x_1, w]\cup P_3'\cup P_2'\cup P_3\cup P_2\cup P_2''\cup P_3''\cup C[w', x_1]$ is then of the type we desired, a contradiction again. This completes the proof of Claim~2. 

\medskip

By Claim~2, we may assume that $u_1\in C(x_1,x_2)$, $u_2\in C(x_2, x_3)$ and $u_3\in C(x_4, x_1)$. Now, we need the stronger version of Perfect's Theorem, Theorem~\ref{thm:strongPT}. Consider $C[x_3,x_4]$ and $H_1=C[u_3, u_2]\cup (\bigcup P_i)$, which are 2-linked. The 2-link consists of two disjoint paths $C[x_4, u_3]$ and $C[u_2, x_3]$.  If $C[x_3,x_4]$ has at least three vertices, apply Theorem~\ref{thm:strongPT} to $C[x_3, x_4]$ and $H_1$. There are three internally disjoint paths  $P_1', P_2'$ and $P_3'$ joining $\{x_3, x_4,w\}\subseteq V(C[x_3,x_4])$ and three distinct vertices $\{u_2,u_3, w'\}\subseteq V(H_1)$. If $C[x_3,x_4]$ is an edge $x_3x_4$, replace $x_3$ by a clique $K$ with $d_G(x_3)\ge 3$ vertices such that the edge of $G$ incident with $x_3$ is incident with exactly one vertex of $K$. Let $G'$ be the new graph which is 3-connected. Then let $C'[x_3,x_4]$ denote a path containing $x_4$ and all vertices of the clique $K$ such that $x_4$ is an endvertex. Then $C'[x_3,x_4]$ has at least four vertices, and $C'[x_3,x_4]$ and $H_1$ are 3-linked in $G'$. Apply Theorem~\ref{thm:strongPT} to $C'[x_3,x_4]$ and $H_1$. There are three internally disjoint paths joining $x_3$ and two vertices of $K$ to three distinct vertices $\{u_2, u_3,w'\}\subset V(H_1)$. The two disjoint paths joining two vertices of $K$ and two vertices from $\{u_2,u_3,w'\}$ in $G'$ correspond to two internally disjoint paths of $G$ joining $x_3$ and two distinct vertices from $\{u_2,u_3, w'\}$. Therefore,  
without loss of generality, assume that $w, x_3$ and $x_4$ are endvertices of $P_1', P_2'$ and $P_3'$ (note that $w$ maybe equal to $x_3$), respectively. 

By symmetry of the two vertices $x_3$ and $x_4$, and the symmetry of $H_1$, the only two possibilities for the other endvertices of $P_1', P_2'$ and $P_3'$ are: 
\begin{center}
 (1) $w'\in P_1', u_2\in P_2'$ and $u_3\in P_3'$; or (2) $u_2\in P_1', w'\in P_2'$ and $u_3\in P_3'$.
\end{center}  

First, assume (1) holds, that $w'\in P_1', u_2\in P_2'$ and $u_3\in P_3'$. By Jumper Operations, $w'\notin \{u_1, x_5\}$. If $w'\in C(x_1, x_2)$, then $G$ has a cycle $C[x_1,w']\cup P_1'\cup C^{-1} [w, x_3]\cup P_2' \cup C^{-1}[u_2, u_1]\cup P_1\cup P_3\cup C[u_3, x_1]$ (if $w'\in C(x_1,u_1)$) or $C[x_1, u_1]\cup P_1\cup P_2\cup C^{-1}[u_2, w']\cup P_1'\cup C[w, x_4]\cup P_3'\cup C[u_3, x_1]$ (if $w'\in C(u_1, x_2)$), either of which violates Claim~1.  If $w'\in \cup_{i=1}^3 P_i$, then consider $x_5$ and the cycle $C[u_3,u_2]\cup P_2'\cup C[x_3,x_4]\cup P_3'$, which violates Claim~2 because $G$ has three disjoint paths joining $x_5$ and the cycle such that one of the path ends at $w\in C[x_3,x_4]$. 

So $w'\in C(x_2, u_2)$ or $w'\in C(u_3,x_1)$. By symmetry, we may assume that $w'\in C(x_2, u_2)$.  
Then applying Perfect's Theorem to $x_4$ and $H_2=H_1\cup C[x_3,w]\cup P_2'\cup P_1'$, there are three internally disjoint paths joining $x_4$ and three distinct vertices of $H_2$ such that two of the paths, say $P_1''$ and $P_2''$, end at $w$ and $u_3$, and the third path $P_3''$ ends at some vertex $w''$ of $H_2$. By Jumper Operations, $w''\notin \{w', u_1, u_2, x_5\}$. An argument  similar to that used in  the proof of $w'\notin C(x_1,x_2)$ shows that $w''\notin C(x_1, x_2)$. By Claim~2, $w''\notin \cup_{i=1}^3 P_i$ and $w''\notin P_2'$ (where we consider $x_5$ and the cycle $C[x_1,w']\cup P_1'\cup C^{-1}[w,x_3]\cup P'_2[x_3,w'']\cup P_3''\cup P_2''\cup C[u_3, x_1]$ for the last case). Clearly, $w''\notin C(u_3, x_1)$ or $C[x_3,w]$ or $P_1'$. Otherwise, $x_4$ could be inserted into the cycle $C[u_3, w']\cup P_1'\cup C^{-1}[w,x_3]\cup P_2'\cup P_2\cup P_3$, which yields a contradiction. So $w''\in C(x_2, w')$ or $w''\in C(w', u_2)$.  
If $w''\in C(x_2,w')$, then $C'=C[x_1,w'']\cup P''_3\cup P_1''\cup C^{-1}[w,x_3]\cup P_2'\cup P_2\cup P_3\cup C[u_3,x_1]$ is a cycle of the type desired, yielding a contradiction.
If $w''\in C(w', u_2)$, then $C'=C[x_1,w'']\cup P_3''\cup P_1''\cup C^{-1}[w,x_3]\cup P_2'\cup P_2\cup P_3\cup C[u_3,x_1]$ also gives a contradiction. This contradiction implies that (1) does not happen.

So, in the following, assume that $u_2\in P_1', w'\in P_2'$ and $u_3\in P_3'$. By Jumper Operations, $w'\notin \{u_1, x_5\}$. If $w'\in C(x_1, u_1)$, then $G$ contains the cycle $C[u_3, w']\cup P_2'\cup C[x_3, w]\cup P_1'\cup C^{-1}[u_2, u_1]\cup P_1\cup P_3$, violating Claim~1. If $w'\in C(u_1, x_2)$, then $G$ contains the cycle $C[x_1, u_1]\cup P_1\cup P_2\cup C^{-1}[u_2, w']\cup P_2'\cup C[x_3, x_4]\cup P_3'\cup C[u_3, x_1]$ of the type desired, yielding a contradiction. By Claim~2, $w'\notin \cup_{i=1}^3 P_i$. If $w'\in C(u_3, x_1)$, then $C'=C[x_1, u_2]\cup P_2\cup P_3\cup P_3'\cup C^{-1}[x_4,x_3]\cup P_1'\cup C[w', x_1]$ is a cycle of the type desired, a contradiction.  So $w'\in C(x_2, u_2)$. Then consider the vertex $x_4$ and the cycle $C''=C[u_3, w']\cup P_2'\cup C[x_3, w]\cup P_1'\cup P_2\cup P_3$. By Perfect's Theorem, there are three internally disjoint paths avoiding $z$ which join $x_4$ and three distinct vertices of the cycle $C''$ such that one of the paths ends at $w\in C''[x_3, x_5]$, but this contradicts Claim~2 (by interchanging the labels of $x_4$ and $x_5$). This completes the proof of the theorem.
\end{proof}

We are now equipped to complete the proof of our main result on cyclability in 3-connected claw-free graphs --- Theorem~\ref{thm:main}. \medskip



\noindent{\bf Proof of Theorem~\ref{thm:main}.}
Let $G$ be a minimum counterexample. Hence, $G$ has a vertex $z$ and  $S=\{x_1, x_2,...,x_5\}$ such that $G$ has no cycle containing all vertices of $S$, but avoiding $z$. By Lemma~\ref{C51-z-in-2cut}, the vertex $z$ belongs to exactly one 3-cut $T$ of $G$.
By Lemma~\ref{lem:z-3-cut}, $T$ separates a single vertex from $S$, say $x_1$.
By Theorem~\ref{thm:c41}, $G$ has a cycle $C$ containing $x_1$ and three other vertices from $S$, but avoiding $z$.
Without loss of generality, assume that $C$ contains $x_1, x_2,x_3$ and $x_4$ in clockwise order. 

By Lemma~\ref{lem:avoid-z-path}, $x_5$ and $C$ are 3-linked in $G-z$ since $z$ is contained in exactly one 3-cut $T$. Hence, there are three internally disjoint paths $P_1, P_2$ and $P_3$ joining $x_5$ and three distinct vertices of $C$. 
Let $u_i$ be the endvertex of $P_i$ on $C$ for $i\in \{1,2,3\}$.
Note that none of the segments of $C$ determined by $x_1,x_2,x_3$ and $x_4$ contains two endvertices of the three paths from $x_5$ to $C$. Otherwise, $x_5$ can be inserted into $C$ to give a cycle of the type we seek, contradicting the assumption that $G$ is a counterexample. 
\medskip
 
\noindent{\bf Claim.}  {\sl The segment $C(x_1, x_2)$ contains one vertex of $\{u_1,u_2, u_3\}$.}  
 \medskip

\noindent {\em Proof of the Claim.}  Suppose not.
Without loss of generality, assume that $u_1\in C[x_2, x_3], u_2\in C[x_3, x_4]$ and $u_3\in C[x_4,x_1]$. Let $H=C[u_2,u_1]\cup(\bigcup P_i)$. By Lemma~\ref{lem:avoid-z-path}, $x_3$ and $H$ are 3-linked in $G-z$. 
Now apply Perfect's Theorem to $x_3$ and $H$ to obtain
three internally disjoint paths joining $x_3$ and three distinct vertices of $H$, two of which, say $P_1'$ and $P_2'$, end at $u_1$ and $u_2$ respectively.
Assume that the third path $P_3'$ from $x_3$ to $H$ ends at $w$. By Jumper Operations, we assume that $w\notin \{u_1,u_2, u_3\}$. 
If $w$ belongs to any of the $P_i-u_i$'s, then there are two internally disjoint paths joining $x_5$ and either $\{u_1,x_3\}$ or $\{u_2,x_3\}$ which can be used to insert $x_5$ into the cycle $C[u_2,u_1]\cup P_1'\cup P_2'$ to yield a cycle of the type desired, a contradiction.  
Similarly, $w\notin C[x_2,u_1]\cup C[u_2, x_4]$, for otherwise $x_3$ can be insert into the cycle $C[u_2,u_1]\cup P_1\cup P_2$ by using two paths from among the three $P_i'$'s to yield a cycle of the type desired, a contradiction. 
If $w\in C(u_3,x_1)$,   the cycle
$C'=C[x_1, u_1]\cup P_1\cup P_3\cup C^{-1}[u_3, u_2]\cup P_2'\cup P_3'\cup C[w, x_1]$ is of the type desired, a contradiction. If $w\in C[x_4,u_3)$, 
the cycle
$C'=C[x_1,u_1]\cup P_1'\cup P_3'\cup C^{-1}[w, u_2]\cup P_2\cup P_3\cup C[u_3,x_1]$  is of the type desired and yet again we have a contradiction. 
Therefore, $w\in C(x_1, x_2)$.

Let $H'=C[u_2,w] \cup (\bigcup P_i)\cup (\bigcup P_i')$. By Lemma~\ref{lem:avoid-z-path}, $x_2$ and $H'$ are 3-linked in $G-z$.  Then apply Perfect's Theorem to obtain
three internally disjoint paths joining $x_2$ and three distinct vertices of $H'$.
Two of these paths, say $P_1''$ and $P_2''$, end at $w$ and $u_1$ respectively, and the third path $P''_3$ ends at some $w'$. Again, by Jumper Operations, assume that $w'\notin\{u_1,u_2,u_3, w\}$. 
A straightforward check confirms that $w'\notin (\bigcup P_i)\cup (\bigcup P_i')\cup C(x_1,w)$,
for otherwise, a cycle of the type sought can be easily constructed.
If $w'\in C(u_2, x_4]$, then $C'=C[x_1,w]\cup P_3'\cup P_2'\cup P_2\cup P_1\cup C^{-1}[u_1,x_2]\cup P_3'' \cup C[w',x_1]$ is a cycle
yielding a contradiction.
If $w\in C(x_4,u_3)$, then
$C'=C[x_1, w]\cup P_3'\cup P_1'\cup C^{-1}[u_1,x_2]\cup P_3'' \cup C^{-1}[w',u_2]\cup P_2\cup P_3 \cup C[u_3,x_1]$, again yields a
contradiction.
So $w\in C(u_3,x_1)$.
But then $C'=C[x_1,w]\cup P_3'\cup P_2'\cup C[u_2, u_3]\cup P_3\cup P_1\cup C^{-1}[u_1,x_2]\cup P_3'' \cup C[w',x_1]$ is a cycle
which yields a contradiction yet again.
This completes the proof of the Claim. \medskip

By the Claim and symmetry,  both $ C(x_1,x_2)$ and $ C(x_4,x_1)$ contain one vertex from $\{u_1,u_2,u_3\}$, which implies that one of $C(x_2,x_3)$ and $C(x_3,x_4)$ does not contain a vertex from $\{u_1,u_2,u_3\}$. By symmetry again, we may assume that it is $C(x_2,x_3)$ which does not contain a vertex from $\{u_1, u_2, u_3\}$ and assume that $u_1\in C(x_1, x_2)$, $u_2\in C(x_3, x_4)$ and $u_3\in C(x_4, x_1)$.
Let  $H=C[u_3,u_2]\cup (\bigcup P_i)$. By Lemma~\ref{lem:avoid-z-path}, $x_4$ and $H$ are 3-linked in $G-z$.
Now apply  Perfect's Theorem to $x_4$ and $H$ 
to obtain three internally disjoint paths from $x_4$ to $H$.
Two of the paths, say $P_1', P_2'$, end at $u_2, u_3$ respectively and the third path $P_3'$ ends  at some vertex $w\notin \{u_1,u_2,u_3\}$ by Jumper Operations. 
First note that $w\notin P_i$ for any $i\in \{1,2,3\}$, for otherwise, $x_5$ can be inserted into the cycle $C'=C[u_3,u_2]\cup (P_1'\cup P_2')$ to generate a cycle which again
yields a contradiction.
Similarly, $w\notin C[x_3, u_2]\cup C[u_3,x_1)$ for otherwise, $x_4$ can be inserted into the cycle $C'=C[u_3, u_1]\cup P_2\cup P_3$ to
form a cycle which yields a contradiction.
If $w\in C(x_1,u_1)$, then cycle $C''=C[x_1,w]\cup P'_3\cup P_2'\cup C^{-1}[u_2,u_1]\cup P_1\cup P_3\cup C[u_3,x_1]$ yields a contradiction.
If $w\in C(u_1,x_2]$, then cycle $C''=C[x_1,u_1]\cup P_1\cup P_2\cup C^{-1}[u_2,w]\cup P_3'\cup P_2'\cup C[u_3,x_1]$ gives a contradiction.
So $w\in C(x_2, x_3)$.
Now consider the cycle $C'=C[u_3,u_2]\cup P_2\cup P_3$ and the vertex $x_4$, which contradicts Claim by interchanging the labels of $x_4$ and $x_5$. This final contradiction completes the proof of Theorem~\ref{thm:main}. \qed

\medskip

Now, we show that Theorem~\ref{thm:main} is sharp by providing infinitely many examples of 3-connected claw-free graphs which are
{\it not} $C(6,1)$.\par

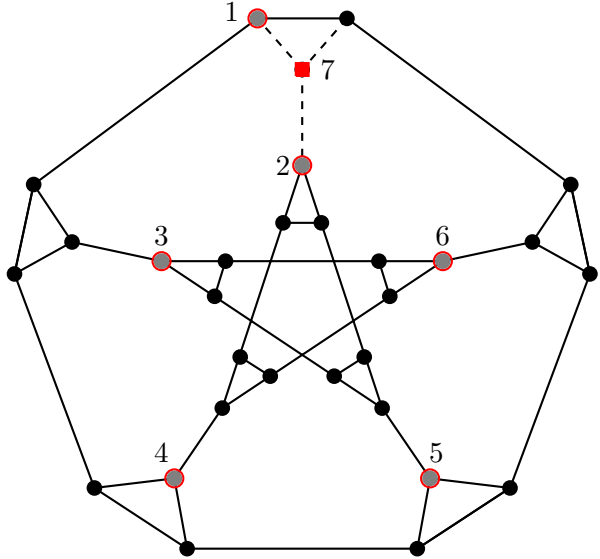
\begin{figure}[!htb] 
\begin{center}
\begin{tikzpicture}[thick, scale=.85]

\draw[] (1.25,1.7)--(4.7, 4.0)--(.3,4.0)--(3.75,1.7)--(2.5,5.5)--(1.25,1.7);
\draw [] (1.53,2.5) -- (2.0,2.2);
\draw [] (3.47,2.5)--(3.0,2.2);
\draw [] (1.3, 4.0) -- (1.13, 3.45);
\draw [] (3.7, 4.0)--(3.87, 3.45);
\draw []  (2.2, 4.6)--(2.8, 4.6);
\draw [] (.5,.6)--(-.75,.45)--(.7,-.5)--(.5,.6);
\draw [] (4.5,.6)--(5.75,.45)--(4.3,-.5)--(4.5,.6);
\draw [] (-2, 3.8) --(-1.7, 5.2) -- (-1.1, 4.3)--(-2,3.8);
\draw [] (7, 3.8)--(6.7, 5.2) -- (6.1, 4.3) --(7,3.8);
\draw [dashed] (2.5, 7.0) -- (1.8,7.8);
\draw [dashed] (3.2,7.8) --(2.5, 7.0);
\draw[] (.5,.6)--(1.25, 1.7);
\draw[] (4.5, 0.6)--(3.75,1.7);

\draw[dashed] (2.5,7.0)-- (2.5, 5.5);

\draw[] (.7,-0.5)--(4.3,-.5)--(5.75,.45)--(7,3.8)--(6.7, 5.2)--(3.2,7.8)--(1.8,7.8)--(-1.7,5.2)--(-2, 3.8)--(-.75,.45); 

\draw[red] (.5,.6) circle (4pt);
\filldraw [gray] (0.5,0.6) circle (3pt); 
\filldraw [] (-.75, .45) circle (3pt);
\filldraw [] (.7, -.5) circle (3pt);

\draw[red] (4.5,.6) circle (4pt); 
\filldraw [gray] (4.5,0.6) circle (3pt); 
\filldraw [] (5.75, .45) circle (3pt);
\filldraw [] (4.3, -.5) circle (3pt);

\filldraw [] (-2, 3.8) circle (3pt);
\filldraw [] (-1.7, 5.2) circle (3pt);
\filldraw [] (-1.1, 4.3) circle (3pt);
\draw[] (-1.1, 4.3)--(.3,4.0);

\filldraw [] (7, 3.8) circle (3pt);
\filldraw [] (6.7, 5.2) circle (3pt);
\filldraw [] (6.1, 4.3) circle (3pt);
\draw [] (6.1,4.3)--(4.7,4.0);

\draw[red] (1.8,7.8) circle (4pt);
\filldraw [gray] (1.8,7.8) circle(3pt);
\filldraw [] (2.5, 7.0) circle (3pt);
\filldraw [] (3.2,7.8) circle(3pt);

\filldraw [] (1.25,1.7) circle (3pt); 
\filldraw [] (1.53,2.5) circle (3pt); 
\filldraw [] (2.0,2.2) circle (3pt);

\filldraw [] (3.75,1.7) circle (3pt);
\filldraw [] (3.47,2.5) circle (3pt); 
\filldraw [] (3.0,2.2) circle (3pt);

\draw[red] (.3,4.0) circle (4pt);
\filldraw [gray] (0.3, 4.0) circle (3pt);
\filldraw [] (1.3, 4.0) circle (3pt);
\filldraw [] (1.13, 3.45) circle (3pt);

\draw[red] (4.7,4.0) circle (4pt);
\filldraw [gray] (4.7, 4.0) circle (3pt);
\filldraw [] (3.7, 4.0) circle (3pt);
\filldraw [] (3.87, 3.45) circle (3pt);

\draw[red] (2.5,5.5) circle (4pt);
\filldraw [gray] (2.5, 5.5) circle (3pt);
\filldraw [] (2.2, 4.6) circle (3pt);
\filldraw [] (2.8, 4.6) circle (3pt);

\filldraw[red] (2.4,6.9)--(2.6,6.9)--(2.6,7.1)--(2.4,7.1)--(2.4,6.9);   

\node [] at (1.4,7.9) { $1$};
\node [] at (2.9, 7) {$7$};
\node [] at (2.2,5.5) {$2$};
\node [] at (.3, 4.4) {$3$};
\node [] at (.3, 1) {4};
\node [] at (4.6, 1) {5};
\node [] at (4.7, 4.4) {6};
 
\end{tikzpicture} 
\caption{\small  A cubic claw-free graph without a cycle through vertices $1,2,...,6$, which avoids  7.}\label{fig:example}
\end{center}
\end{figure}

Let $G$ be the 3-connected claw-free graph on thirty vertices obtained by replacing each of the ten vertices of the Petersen graph
with a triangle.  (See Figure~\ref{fig:example}.) It is easy to check that there is no cycle in this graph containing the six vertices numbered 1 through 6 which fails to contain the seventh vertex labeled 7.
Hence this graph does not possess the property $C(6,1)$.
To obtain infinitely many such counterexamples, one may simply replace any of the triangles with a larger complete graph.

It is also interesting to observe that the graph shown in Figure~\ref{fig:example} is cubic as well.
Hence 3-connected claw-free cubic graphs do not necessarily possess the property $C(6,1)$ either. 
As we mentioned in the Introduction, a 3-connected cubic graph may not be $C(3,1)$, an example being $K_{3,3}$.

\section{Graphs on Surfaces}

Tutte \cite{T} proved that every 4-connected plane graph $G$ is Hamiltonian, and hence is $C(n,0)$ where $n=|V(G)|$. However, for 3-connected plane graphs, the maximum cyclability is only 5 and this bound is sharp in the sense that there exist  3-connected plane graphs which are not $C(6,0)$ \cite{PW, S}. The same fact holds true  for 3-connected plane triangulations. There are infinitely many  3-connected plane triangulations which are not $C(6,0)$ and which are not $C(4,1)$. For example, in the eleven-vertex graph shown in Figure~\ref{fig:triangulation}, there is no cycle through vertices 1,2,\ldots, 6, and neither is there a cycle through vertices 1, 2, 3 and 4, which avoids 7. Note that the example in Figure~\ref{fig:triangulation} can be extended to infinitely many examples by repeatedly adding a new vertex in the exterior face and connecting it to the three vertices of the exterior triangle.

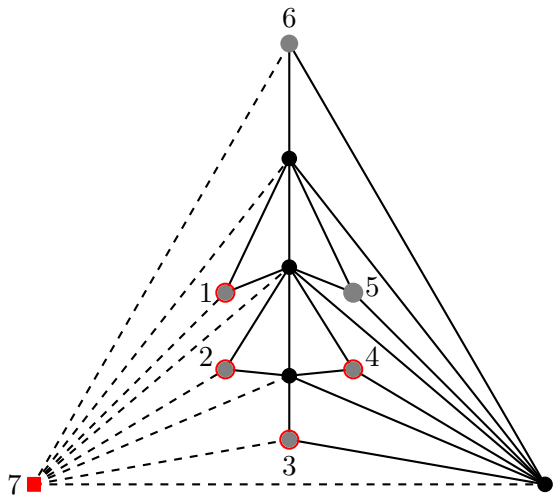
\begin{figure}[!htb] 
\begin{center}
\begin{tikzpicture}[thick, scale=.85]

  \draw[dashed] (0,0)--(8,0);
\draw[dashed](4, 6.9)--(0,0);
\draw[] (8,0)--(4, 6.9);
   \draw[] (4,.7)--(4,6.9); 
    \draw[dashed] (0,0)--(4,1.7);
\draw[dashed] (0,0) --(4, 3.4);
\draw[dashed] (0,0)-- (4, 5.1);
\draw[dashed] (0,0) --(3, 1.8);
\draw[dashed] (0,0)-- (3, 3);
\draw[] (3,1.8) --(4, 3.4)--(5, 1.8)--(8,0);
\draw[] (3,1.8)-- (4, 1.7)--(5, 1.8);
\draw[] (3,3) --(4, 3.4)--(5, 3)--(8,0);
\draw[] (3,3)-- (4, 5.1)--(5, 3);
 \draw[dashed] (0,0)--(4, .7);
\draw[] (4, .7)--(8,0);
 \draw[] (8,0)--(4,1.7);
\draw[] (8,0) --(4, 3.4);
\draw[] (8,0)-- (4, 5.1);

\filldraw [] (8,0) circle (3pt);
\filldraw [gray] (4, 6.9) circle (3.5pt);
\filldraw [] (4, 1.7) circle (3pt);
\filldraw [] (4, 3.4) circle (3pt);
\filldraw [] (4, 5.1) circle (3pt);

\draw[red] (4,.7) circle (4pt);
\filldraw[gray] (4, .7) circle (3pt);

\draw[red] (3,1.8) circle (4pt);
\filldraw[gray] (3, 1.8) circle (3pt);

\draw[red] (3,3) circle (4pt);
\filldraw[gray] (3, 3) circle (3pt);

\draw[red] (5,1.8) circle (4pt);
\filldraw[gray] (5, 1.8) circle (3pt);
\filldraw[gray] (5, 3) circle (4pt);

\filldraw[red] (-.1,.1)--(.1,.1)--(.1,-.1)--(-.1,-.1);   

\node [] at (2.7,3) { $1$};
\node [] at (2.7, 2) {$2$};
\node [] at (4, .3) {$3$};
\node [] at (5.3, 2) {$4$};
\node [] at (5.3, 3.1) {$5$};
\node [] at (4,7.3) {$6$};
\node [] at (-.3, 0) {$7$}; 
 
\end{tikzpicture} 
\caption{\small A plane triangulation without a cycle through vertices $1,2,3,4$, which avoids $7$.}\label{fig:triangulation}
\end{center}
\end{figure}

In fact, for {\em any} given closed surface $\Sigma$, there are infinitely many graphs which are neither $C(6,0)$ nor $C(4,1)$, even for surface triangulations which have the maximal edge density among the graphs embedded in the surface. To construct a surface triangulation which is neither $C(6,0)$ nor $C(4,1)$: 
take any surface triangulation of $\Sigma$, then glue the exterior triangle of the graph in Figure~\ref{fig:triangulation} to any face-bounding triangle of the triangulation to generate a new triangulation. Then the new triangulation has no cycle through the six gray vertices, nor does it have a cycle through 1,2,3 and 4, which avoids 5. Note that, a surface triangulation is always a polyhedral embedding. Hence, there are infinitely many graphs polyhedrally embedded in any closed surface which do not have the properties $C(6,0)$ or $C(4,1)$. 

On the other hand, Theorem~\ref{thm:poly} shows that a graph polyhedrally embedded in a closed surface must have the property $C(3,1)$. Note that, the assumption that the embedding be polyhedral is necessary because $K_{3,3}$ has a closed 2-cell embedding in the projective plane but does not have the property $C(3,1)$. A graph $G$ admitting a polyhedral embedding is 3-connected and the neighbors of any vertex $v\in V(G)$ belong to a cycle of $G-v$ (i.e., the symmetric difference of the face boundaries containing $v$) (see \cite{MT}).   

Theorem~\ref{thm:poly} follows directly from the following more general result.

\begin{thm}\label{thm:general-C(3,1)}
Let $G$ be a 3-connected graph and $z\in V(G)$ such that $G-z$ has a cycle $C_z$ containing all neighbors of $z$. Then, for any other three vertices $x_1, x_2$ and $x_3$, $G$ has a cycle passing through $x_1, x_2, x_3$ but avoiding $z$. 
\end{thm}
\begin{proof}
 Suppose to the contrary that $G$ is a counterexample. Then $G$ has vertices   $x_1, x_2, x_3$ such that $G$ does not have a cycle through $x_1, x_2$ and $x_3$, which avoids $z$.  
Since $G$ is 3-connected, for any pair of vertices $x_i$ and $x_j$ from $\{x_1,x_2,x_3\}$, there is a cycle $C_{ij}$ containing $x_i$ and $x_j$. Because $G$ is a counterexample, $C_{ij}$ does not contain the third vertex $x_k$. On the other hand, $G$ has three internally disjoint paths $P_1, P_2$ and $P_3$ from $x_k$ to the cycle $C_{ij}$ ending at three different vertices of $C_{ij}$ by Menger's Theorem. Again, since $G$ is a counterexample, one of these three paths must contain the vertex $z$. Otherwise, two of $P_1, P_2$ and $P_3$ must end on the same segment of $C_{ij}$ separated by $x_i$ and $x_j$. These two paths then could be used to insert $x_k$ into the cycle $C_{ij}$ to generate a cycle of the type we seek, a contradiction. Note that, if at most two of the three paths $P_1, P_2$ and $P_3$ intersect $C_z$, then $G$ has three paths from $x_k$ to $C_{ij}$ by using the segments of $C_z$ as a detour to avoid the vertex $z$, a contradiction again. 
Hence, we prove the following claim. \medskip

\noindent{\bf Claim~1.} {\em Either $x_k$ belongs to $C_z$ or the three paths $P_1,P_2$ and $P_3$ intersect $C_z$ at three different vertices.} \medskip

Again since $G$ is a counterexample, $C_z$ contains at most two vertices from $\{x_1,x_2,x_3\}$. If $|C_z\cap \{x_1,x_2,x_3\}|=2$, then $G$ has three internally disjoint paths from the vertex $x_i$ not on $C_z$ to $z$ which intersects $C_z$ at three different vertices. Further, the vertex $x_i$ can be inserted into $C_z$ using two of the three paths to generate a cycle of the type desired, a contradiction. Hence  $C_z$ contains at most one vertex from $\{x_1,x_2,x_3\}$. Without loss of generality, assume that $x_1, x_2\notin C_z$.
Since $G$ is 3-connected, there are three internally disjoint paths joining each of  $x_1$ and $x_2$  to $z$. Assume that the three paths from $x_1$ to $z$ intersect $C_z$ at $v_1, v_2$ and $v_3$, and  denote the segment from $x_1$ to $v_i$ by $P[x_1,v_i]$ for $i\in \{1,2,3\}$. Similarly, there are three internally disjoint paths from $x_2$ to $C_z$ ending at three different vertices $u_1, u_2$ and $u_3$. Denote these paths by $P[x_2, u_i]$ for $i\in\{1,2,3\}$. Let $\Gamma_1=\{v_1,v_2,v_3\}$ and $\Gamma_2=\{u_1,u_2,u_3\}$.

Note that either $x_3\in C_z$ or $x_3\notin C_z$. First, we consider the case that $x_3\in C_z$. Any two paths of the type $P[x_1, v_i]$, say $P[x_1,v_1]$ and $P[x_1,v_2]$,  together with the segment of $C_z$ separated by  $v_1$ and $v_2$ which contains $x_3$ form a cycle. Let us denote such a cycle  by $C_{13}$. Then by Claim~1, $P[x_2, u_i]$ does not internally intersect $P[x_1,v_j]$ ($u_i$ and $v_j$ may be the same), for $i,j \in \{1,2,3\}$. Otherwise, there are three internally disjoint paths from $x_2$ to $C_{13}$ which do not intersect $C_z$ at three different vertices, a contradiction to Claim~1. Since $x_3\in C_z$, 
without loss of generality, assume that $v_1,v_2$ and $v_3$ appear in clockwise order on $C_z$ and $x_3\in C_z[v_1, v_2]$. By symmetry, assume that $|C_z(v_3,v_1)\cap \Gamma_2|\le |C_z(v_2, v_3]\cap \Gamma_2|$. Then $|C_z(v_3,v_1)\cap \Gamma_2|\le 1$ because $|\Gamma_2|=3$. If $|C_z(v_3,v_1)\cap \Gamma_2|=1$, say $u_3\in C_z(v_3,v_1)$, then $C_z(v_2,v_3]$ has a vertex from $\Gamma_2$, say $u_2$. It follows that $G$ has a cycle $C=P[x_1,v_1]\cup C[v_1,u_2]\cup P[x_2,u_2]\cup P[x_2,u_3]\cup C_z^{-1}[u_3, v_3]\cup P[x_1,v_3]$ which contains $x_1, x_2$ and $x_3$, but not $z$, a contradiction. Hence, $C_z(v_3,v_1)\cap \Gamma_2=\emptyset$. So $\Gamma_2\subset C_z-C_z(v_3,v_1)=C_z[v_1,v_3]$. Then one of segments $C_z[v_1,x_3)$ and $C_z[x_3,v_3]$ contains two vertices from $\Gamma_2$, say $u_1$ and $u_2$. But then $x_2$ can be inserted into the cycle $P[x_1,v_1]\cup C_z[v_1,v_3]\cup P[x_1,v_3]$ using the two paths $P[x_2, u_1]$ and $P[x_2,u_2]$ to yield a cycle of the type we seek, a contradiction again. This contradiction implies that $x_3\notin C_z$. 

Let $H=C_z\cup (\bigcup P[x_1, v_i])\cup (\bigcup P[x_2, u_j])$. It is easily seen that every path of the type $P[x_1,v_i]$ or $P[x_2,u_j]$ is contained in a cycle of $H$ which contains both $x_1$ and $x_2$, but not $z$. Since $G$ is a counterexample and $x_3\notin C_z$, it follows that $x_3\notin H$. Since $G$ is 3-connected, there are three internally disjoint paths from $x_3$ to $H$ ending at three different vertices $w_1, w_2$ and $w_3$ by Menger's Theorem. Let $\Gamma_3=\{w_1, w_2, w_3\}$.  By Claim~1, $\Gamma_3\subset C_z$. 
 
\medskip

\noindent{\bf Claim~2.} {\em For every two vertices $v,v'\in \Gamma_i$ with $i\in \{1,2,3\}$, both $C_z(v,v')$ and $C_z(v',v)$ contain a vertex $v''\in \Gamma_j$ for some $j\in \{1,2,3\}\backslash \{i\}$.} \medskip

\noindent{\em Proof of Claim~2.} Suppose that Claim~2 does not hold. Without loss of generality, assume that $C(v,v')\cap \Gamma_j=\emptyset$ for every $j\in \{1,2,3\}\backslash \{i\}$. Let $C_z'=C_z-C_z(v,v')\cup P[x_i,v]\cup P[x_i,v']$. Then $x_i\in C_z'$ and $\Gamma_j\subseteq C_z'$ for every $j\in \{1,2,3\}\backslash \{i\}$. Treat $C_z'$ as $C_z$ and $x_i$ as $x_3$ in the case that $x_3\in C_z$. An argument similar to the one used in the proof for the case $x_3\in C_z$ shows that $G$ must have a cycle containing $x_1, x_2$ and $x_3$, but not $z$, which contradicts the assumption that $G$ is a counterexample. This completes the proof of Claim~2. \medskip

Note that either $C_z(v_1,v_2]$ or $C_z(v_2,v_1]$ contains two vertices from $\Gamma_2$. Without loss of generality, assume that $\Gamma_2$ has two vertices in $C_z(v_2,v_1]$, namely $u_{1}$ and $u_2$ (relabeling if necessary). Assume that $v_1, v_2, u_1$ and $u_2$ appear in clockwise order on $C_z$. Consider the four vertices $v_1,v_2,u_1$ and $u_2$ on $C_z$, and observe that the segment $C_z(v_1,u_1]$ is symmetric to the segment $C_z(u_1,v_1]$. Then either $C_z(v_1,u_1]$ or $C_z(u_1, v_1]$ contains two vertices from $\Gamma_3$. By the symmetry of $C(v_1,u_1]$ and $C(u_1,v_1]$, assume that $|\Gamma_3\cap C_z(v_1,u_1]|\ge 2$, and $w_1, w_2\in C_z(v_1, u_1]$ such that $v_1, w_1$ and $w_2$ appear in clockwise order on $C_z$.

If $w_2\in C_z(v_2, u_1]$, then $x_3$ can be inserted into the cycle  $C_z[u_2,v_1]\cup P(x_1,v_1)\cup P(x_1,v_2)\cup C_z[v_2,u_1]\cup P(x_2,u_1)\cup P(x_2,u_2)$ using the two paths $P(x_3,w_1)$ and  $P(x_3, w_2)$ if $w_1\in C_z[v_2,u_1]$ (or $P(x_3,w_1)\cup C_z[w_1,v_2]$ if $w_1\in C_z(v_1,v_2)$) to give a cycle of the type we seek, and again we have a contradiction to the assumption that $G$ is  a counterexample. This contradiction implies that both $w_1$ and $w_2\in C_z(v_1,v_2]$. 

 By Claim~2, $C_z(u_1, u_2)$ contains a vertex from $\Gamma_1\cup \Gamma_3$. If $w_3\in C_z(u_1,u_2)$, then $x_3$ can be inserted into the cycle  $C'=C_z[u_2,v_1]\cup P(x_1,v_1)\cup P(x_1,v_2)\cup C_z[v_2,u_1]\cup P(x_2,u_1)\cup P(x_2,u_2)$ using the two paths   $P(x_3,w_2)\cup C_z[w_2,v_2]$  and  $P(x_3, w_3)\cup C_z^{-1}[w_3,u_1]$, a contradiction. Hence, $v_3\in C_z(u_1,u_2)$. But then the cycle $P[x_1,v_1]\cup C_z[v_1,w_1]\cup P[x_3,w_1]\cup P[x_3,w_2]\cup C[w_2,u_1]\cup P[x_2,u_1]\cup P[x_2,u_2]\cup C^{-1}_z[u_2, v_3]\cup P[x_1,v_3]$ is a cycle of the type desired, a contradiction. This final contradiction completes the proof. 
\end{proof}

\section{Concluding Remarks}

It is worth mentioning that Kelmans and Lomonosov \cite{KL} characterized $k$-connected graphs without the property $C(k+2,0)$ for integer $k\ge 2$, and Watkins and Mesner \cite{WM} characterized 2-connected graphs without the property $C(3,0)$. It might be possible to use these characterizations to give alternative proofs of Theorem 4.2 (using Kelmans and Lomonosov's characterization for $k=2$) and Theorem~\ref{thm:general-C(3,1)} (using Watkins and Mesner's characterization).  However, these characterizations are not simple, and the applications of them are not straightforward.

Let $f_{\mathcal F}( k, t)$ be the largest integer $m$ such that every $k$-connected graph in the family $\mathcal F$ is $C(m,t)$ where $t, k$ are integers with $k\ge t+2$. If every $k$-connected graph in $\mathcal F$ is Hamiltonian, define $f_{\mathcal F}(k,0)=\infty$. Since a $k$-connected graph has a cycle through any given $k$ vertices (cf. \cite{D}), it follows immediately that $f_{\mathcal F}(k,t)\ge k-t$.  But if $\mathcal F$ contains $K_{k,k}$, then  $f_{\mathcal F}(k,t)= k-t$. For some interesting families of graphs, $f_{\mathcal F}(k,t)$ could be bigger than $k-t$. For example, if we let $\mathcal F$ be the family of claw-free graphs, then $f_{\mathcal F}(3,1)=5$ (Theorem 1.3); if we let $\mathcal F$ be the family of polyhedral maps, then $f_{\mathcal F}(3, 1)=3$ (Theorem~\ref{thm:poly}). Particularly, if let $\mathcal F$ be the family of plane graphs, then $f_{\mathcal F}(3, 0)=5$ (\cite{PW,S}), and $f_{\mathcal F}( 4, 0)=f_{\mathcal F}(4,1)=f_{\mathcal F}(4,2)=\infty$ (\cite{T, DN, TY}) which implies $f_{\mathcal F}(5,t)=\infty$ for any $t\in \{0,1,2,3\}$. Note that, the connectivity of plane graphs is at most 5. Therefore, the exact value of $f_{\mathcal F}(k,t)$ for plane graphs has been determined.  It would be interesting to study $f_{\mathcal F}(k,t)$ for other families of graphs as well.

The cyclability for  graphs embedded in surfaces also deserves to be further explored.  It would be interesting to determine $f_{\mathcal F} (k,t)$ when $\mathcal F$ is the family of graphs embedded in a closed surface $\Sigma$, or the family of triangulations of a closed surface $\Sigma$.  
Note that, with only finitely many exceptions, a graph embedded in a surface has average degree less than 7. Therefore,  if $\mathcal F$ is an infinite family of $k$-connected graphs embedded in a surface, then $k$ could only be a small integer between 2 and 6. 

\bigskip

\noindent{\bf Acknowledgment.} The authors would like to thank the referee for valuable comments.


\end{document}